\documentclass[12pt,oneside]{amsart}
\usepackage{hyperref} 
\usepackage{amssymb}
\usepackage{euscript}
\usepackage{pstricks}
\def\calT{\EuScript{T}} 
\newcommand{\C}{\mathbb{C}} 
\newcommand{\Z}{\mathbb{Z}} 
\newcommand{\cp}{\C P} 
\newcommand\R{\mathbb{R}} 
\newcommand{\Etilde}{\tilde{E}} 
\newcommand{\calO}{\mathcal{O}} 
\renewcommand{\O}{{\rm O}} 

\newcommand{\Mhat}{M\llap{$\widehat{\phantom{N}}$}} 
\newcommand{\Wtilde}{\widetilde{W}} 
\renewcommand{\H}{\mathcal{H}} 
\newcommand{\E}{\mathcal{E}} 
\newcommand{\D}{\Delta} 
\newcommand{\frakX}{\mathfrak{X}} 
\newcommand{\Y}{\mathfrak{Y}} 
\newcommand{\calI}{\mathcal{I}} 
\def\CAT#1{CAT(#1)}
\def\x{\chi} 
\newcommand\CATx{\CAT{$\x$}}

\newcommand{\calC}{\mathcal{C}} 
\newcommand{\calD}{\mathcal{D}} 

\renewcommand{\epsilon}{\varepsilon}
\def\xtilde{\tilde{x}}
\def\Utilde{\widetilde{U}}

\DeclareMathOperator{\Stab}{\hbox{\rm Stab}} 
\DeclareMathOperator{\diam}{\hbox{\rm diam}}
\DeclareMathOperator{\circum}{\hbox{\rm circum}}
\DeclareMathOperator{\diag}{\hbox{\rm diag}}
\DeclareMathOperator{\Zykel}{\hbox{\rm Zykel}} 
\def\Pic{\mathop{\rm Pic}\nolimits} 
\def\Aut{\mathop{\rm Aut}\nolimits} 
\newcommand\semidirect{\rtimes}
\newcommand\tensor{\otimes}
\newcommand\sset{\subseteq}
\newcommand\iso{\cong}
\newcommand\homotopic{\simeq}

\newcommand{\marginlabel}[1]{\mbox{}\marginpar{\raggedleft\hspace{0pt}#1}}
\newcommand{\defn}[1]{\vardefn{#1}{#1}}
\newcommand{\vardefn}[2]{\emph{#1}\marginlabel{\footnotesize #2}}
\newcommand{\notation}[1]{#1\marginlabel{#1}}
\newcommand{\nonotation}[1]{\marginlabel{#1}}

\theoremstyle{plain}
\newtheorem{theorem}{Theorem}
\newtheorem{lemma}[theorem]{Lemma}
\newtheorem{conjecture}[theorem]{Conjecture}
\newtheorem{corollary}[theorem]{Corollary}
\numberwithin{theorem}{section}
\numberwithin{equation}{section}

\theoremstyle{remark}
\newtheorem{example}[theorem]{Example}
\newtheorem*{remark}{Remark}
\newtheorem*{remarks}{Remarks}

\renewcommand{\theenumi}{\alph{enumi}}

\begin{document}

\title[Completions and branched covers]{Completions, branched covers, Artin groups and
  Singularity Theory}
\author{Daniel Allcock}
\address{Department of Mathematics\\University of Texas at Austin\\Austin, TX 78712}
\email{allcock@math.utexas.edu}
\urladdr{http://www.math.utexas.edu/\textasciitilde allcock}
\date{June 17, 2011}
\thanks{Partly supported by NSF grants DMS-024512 and DMS-0600112.}
\begin{abstract}
We study the curvature of metric spaces and branched covers of
Riemannian manifolds, with applications in topology and algebraic
geometry.  Here curvature bounds are expressed in terms of the
\CATx\ inequality.  We prove a general \CATx\ extension theorem,
giving sufficient conditions on and near the boundary of a locally
\CATx\ metric space for the completion to be \CATx.  We use this to
prove that a branched cover of a complete Riemannian manifold is
locally \CATx\ if and only if all tangent spaces are \CAT0 and the
base has sectional curvature bounded above by $\x$.  We also show that
the branched cover is a geodesic space.  Using our curvature bound and
a local asphericity assumption we give a sufficient condition for the
branched cover to be globally \CATx\ and the complement of the branch
locus to be contractible.

We conjecture that the universal branched cover of $\C^n$ over the
mirrors of a finite Coxeter group is \CAT0.  Conditionally on this
conjecture, we use our machinery to prove the Arnol$'$d-Pham-Thom
conjecture on $K(\pi,1)$ spaces for Artin groups.  Also conditionally,
we prove the asphericity of moduli spaces of amply lattice-polarized
K3 surfaces and of the discriminant complements of all the unimodal
hypersurface singularities in Arnol$'$d's hierarchy.
\end{abstract}

\maketitle

\section{Introduction}
\label{sec-intro}

\noindent
We are interested in when a branched cover $M'$ of a complete
Riemannian manifold $M$ with sectional curvature${}\leq\x$ satisfies
the same curvature bound.  Our interest in this problem goes back to
\cite{Allcock-asphericity} and stems from its
applications to the well-known $K(\pi,1)$ problem for Artin groups,
the topology of certain moduli spaces of K3 surfaces, and
the topology of discriminant complements of singularities.  A simple
way to formulate the idea of a branched cover in this setting is to
remove a subset $\D$ from $M$, leaving $M_0$, take a covering space
$M_0'$ of $M_0$, and metrically complete it to get $M'$.  So the study
of this sort of branched covering is naturally phrased in terms of
completions of metric spaces that satisfy some local curvature bounds.

A good setting for this situation uses the \CATx\ inequality of as the
definition of an upper curvature bound.  So we address a question more
general than the original question about Riemannian manifolds.
Namely, if $X$ is a metric space, $\D$ a closed subset of it, and
$X-\D$ is locally \CATx, under what conditions is $X$ \CATx?  Here is
our result when $\x\leq0$; the corresponding result for $\x>0$ is
theorem~\ref{thm-positive-curvature-case}.

\renewcommand{\theenumi}{\Alph{enumi}}

\begin{theorem}[=\ref{thm-CATx-extension}; the ``\CATx\ extension theorem'']
\label{thm-cat-x-extension-INTRO}
Suppose $\x\leq0$, \notation{$X$} is a complete geodesic space and
\notation{$\D$} a nonempty closed convex subset.  Assume also:
\begin{enumerate}
\item
\label{hyp-radial-geodesic-triangles-CATx-INTRO}
every geodesic triangle with a vertex in $\D$ satisfies \CATx;
\item
\label{hyp-tangent-spaces-unique-local-geodesics-INTRO}
local geodesics in the metric completion $\overline{T_cX}$ of the
tangent space $T_c X$
are unique, for every $c\in \D$; and
\item
\label{hyp-quantified-local-CATx-INTRO}
there exists \notation{$\lambda$}${}>0$ such that for all $x\in X-\D$, the
  closed ball with center $x$ and radius $\lambda\cdot d(x,\D)$ is
  complete and \CATx.
\end{enumerate}
Then $X$ is \CATx.
\end{theorem}

\renewcommand{\theenumi}{\alph{enumi}}

\noindent
The hypotheses \eqref{hyp-radial-geodesic-triangles-CATx-INTRO} and
\eqref{hyp-tangent-spaces-unique-local-geodesics-INTRO} are obviously necessary, but
they are not sufficient.  We think of
\eqref{hyp-quantified-local-CATx-INTRO} as a sort of ``uniformly local
\CATx'' condition.  To us this seems a
simple and natural condition, particularly in light of our
example~\ref{eg-cusped-cone}.

One application of this theorem and its $\x>0$ analogue is to a
simplicial complex $X$ in which each simplex is given a metric of
constant curvature~$\x$.  The key result here is that $X$ is locally
\CATx\ just if every link in \CAT1.  This was stated by Gromov
\cite[4.2.A]{Gromov} and first proven in full generality by Bridson
\cite{Bridson-thesis}\cite[II.5]{bridson-haefliger}.  (Any of several minor regularity
conditions is needed in order for $X$ to be a geodesic space.  Also,
see \cite[p.~377]{Bridson-thesis} for references to earlier work.)
using our results one can reprove Bridson's theorem by an induction on
codimension (similar to but easier than our proof of
theorem~\ref{thm-riem-mfld-br-cover-locally-CATx}).  The link
condition is exactly the condition that the tangent spaces be \CAT0,
and our hypothesis~\eqref{hyp-quantified-local-CATx-INTRO} is trivial in this case.  So one
can view our theorem as allowing the link condition to apply to metric
spaces more general than metrized simplicial complexes.

Our theorem also promises to have applications to
situations involving degenerate Riemannian metrics.  For example, a
celebrated result of Wolpert \cite{Wolpert}\cite{Wolpert-survey} is
that Teichm\"uller space $\calT$ is \CAT0 under the Weil-Petersen
metric.  Although this metric is Riemannian with nonpositive sectional
curvature, it is not complete, so it was not clear that geodesics
exist.  The geometry of the metric completion near the boundary is
similar to our example~\ref{eg-cusped-cone}, so it should be possible
to develop the Weil-Petersen geometry of Teichm\"uller space using
our machinery.

Our main result on the curvature of branched covers of Riemannian
manifolds is the expected one:

\begin{theorem}[=\ref{thm-riem-mfld-br-cover-locally-CATx}]
\label{thm-riem-mfld-br-cover-locally-CATx-INTRO}
Suppose $M$ is a complete Riemannian manifold with section
curvature${}\leq\x\in\R$ and $\D\sset M$ is locally the union of
finitely many totally geodesic submanifolds of codimension~$2$.
Suppose also that $M_0'$ is a covering space of $M_0:=M-\D$ and $M'$
is its metric completion.  Then $M'$ is locally \CATx\ if and only if
each of its tangent spaces is \CAT0.
\end{theorem}

\noindent
This reduces the question of local curvature bounds to an
infinitesimal question.  The point is that the tangent spaces to $M'$
are branched covers of the tangent spaces to $M$
(lemma~\ref{lem-tangent-spaces-are-branched-covers-of-tangent-spaces}),
which are Euclidean.  In some examples one can verify this
tangent-space condition, but in general it seems to be quite hard.
Our proof of the theorem is also surprisingly hard, requiring a
delicate double induction.  Charney and Davis \cite[thm.~5.3]{CD}
state a result similar to this one, but there is a gap in the proof;
see the remark at the end of section~\ref{sec-riem-mflds-local}.

Next we study the global geometry of $M'$.  First we show that $M'$ is
a geodesic space (theorem~\ref{thm-br-cover-is-geodesic-space}).  Then
we prove the following theorem, which is our tool in applications.
(We say a space is \defn{aspherical} if its homotopy groups
$\pi_{n>1}$ are trivial.)

\begin{theorem}[=\ref{thm-riem-mfld-br-cover-globally-CATx}]
\label{thm-thm-riem-mfld-br-cover-globally-CATx-INTRO}
Assume the hypotheses of
theorem~\ref{thm-riem-mfld-br-cover-locally-CATx-INTRO}, with
$\x\leq0$, $M$ connected and $M_0'$ the universal cover of $M_0$.  If
$T_xM-T_x\D$ is aspherical for all $x\in X$, and each tangent space to
$M'$ is \CAT0, then $M_0'$ is contractible and $M'$ is \CATx.
\end{theorem}

This theorem generalizes of the main result of
\cite[lemma~3.3]{Allcock-asphericity}, which is the special case where
$\D$ is locally modeled on the coordinate hyperplanes of $\C^n$.

We give three applications of this theorem, but unfortunately they are
conditional on the following conjecture about finite Coxeter groups.
Nevertheless, we view the unification of these problems and their
reduction to this conjecture as progress.

\begin{conjecture}[=\ref{conj-Coxeter-arrangement-implies-CAT0}]
\label{conj-Coxeter-arrangement-implies-CAT0-INTRO}
Let $W$ be a finite Coxeter group acting isometrically on
$\C^n$ and let $\D$ be the union of the hyperplanes fixed by the
reflections in $W$.  Then the metric completion of the universal cover
of $\C^n-\D$ is \CAT0.
\end{conjecture}

This is closely related
to conjecture~3 of Charney and Davis \cite{CD-artin}, and in fact our
approach to the Artin group $K(\pi,1)$ problem is to show that our
conjecture implies theirs.  Combining our conjecture with Deligne's celebrated result
\cite{Deligne} on the asphericity of mirror complements for finite
Coxeter groups, we obtain:

\begin{corollary}[=\ref{cor-assuming-conjecture-locally-Coxeter-implies-cover-contractible}]
\label{cor-assuming-conjecture-locally-Coxeter-implies-cover-contractible-INTRO}
Assume conjecture~\ref{conj-Coxeter-arrangement-implies-CAT0-INTRO}.
Also assume the situation of
theorem~\ref{thm-riem-mfld-br-cover-locally-CATx-INTRO}, with
$\x\leq0$, $M$ connected and $\D$ locally modeled on the complexified
mirror arrangements of finite Coxeter groups.  Then the universal
cover of $M-\D$ is contractible.
\end{corollary}

We have already mentioned that the Arnol$'$d-Pham-Thom conjecture on
$K(\pi,1)$ spaces for Artin groups follows from conjecture~\ref{conj-Coxeter-arrangement-implies-CAT0-INTRO}:

\begin{theorem}[=\ref{thm-Arnold-Pham-Thom-conjecture}]
\label{thm-Arnold-Pham-Thom-conjecture-INTRO}
Assume conjecture~\ref{conj-Coxeter-arrangement-implies-CAT0-INTRO}.  Let
$W$ be any Coxeter group, acting on its open Tits cone $C\sset\R^n$,
and let $M$ be its tangent bundle $TC$.  Let $\D$ be the union of the
tangent bundles to the mirrors of the reflections of $W$.  Then $M-\D$
has contractible universal cover.
\end{theorem}

In our other applications we will take $M$ to be the Hermitian
symmetric space $P\Omega$ associated to the orthogonal group
$\O(2,n)$.  The first application is to the moduli spaces of amply
lattice-polarized K3 surfaces; here $K$ is the ``K3 lattice''
$E_8^2\oplus\bigl(\begin{smallmatrix}0&1\\1&0\end{smallmatrix}\bigr)^3\iso
  H^2(\hbox{any K3 surface};\Z)$.

\begin{theorem}[=\ref{thm-lattice-polarized-K3}]
\label{thm-lattice-polarized-K3-INTRO}
Assume conjecture~\ref{conj-Coxeter-arrangement-implies-CAT0-INTRO}.
Suppose $M$ is an integer quadratic form of signature $(1,t)$ with a
fixed embedding in $K$.  Then the moduli space of amply $M$-polarized
K3 surfaces $(X,j)$, for which the composition $M\to\Pic X\to H^2(X)$ is
isomorphic to $M\to K$, has contractible orbifold universal cover.
\end{theorem}

For orientation we remark that a K3 amply polarized by the
$1$-dimensional lattice $\langle4\rangle$ is the same thing as a
smooth quartic surface in $\cp^3$.  The global Torelli theorem for
lattice-polarized K3s says that the spaces of amply lattice-polarized
K3s are exactly the sort of space to which our techniques apply.  A
similar situation arises in Bridgeland's study of stability conditions
on K3 surfaces \cite{Bridgeland}; see the remark after
theorem~\ref{thm-lattice-polarized-K3}.

Our final application is to the discriminant complements of
singularities.  The discriminant complement of a singularity is
essentially the space of all deformations of the singularity that are
deformed enough for the singularity to become smooth.  (See
section~\ref{sec-singularity-theory} for precise definitions.)  The
nature of the discriminant complement has been central to singularity
theory since Brieskorn's famous paper \cite{Brieskorn-ADE} on the
discriminants of the simple ($A_n$, $D_n$ and $E_n$) singularities.
In Arnol$'$d's famous hierarchy of hypersurface singularities, the
singularities one step more complicated than the simple ones are the
``unimodal'' singularities.  There are three kinds: simply-elliptic,
cusp and exceptional.

\begin{theorem}[=\ref{thm-ellipic-and-exceptional-aspherical}+\ref{thm-cusp-discriminant-complement-aspherical}]
\label{thm-unimodal-singularities-INTRO}
Assume conjecture~\ref{conj-Coxeter-arrangement-implies-CAT0-INTRO}. Then the discriminant complement of any
unimodal hypersurface singularity is aspherical.
\end{theorem}

The corresponding theorem for simple singularities is due to Deligne
\cite{Deligne} and we use his result in our proof.  We also rely on a
great deal of work by Looijenga: his work provides descriptions of the
discriminant complements to which our techniques can be adapted.  One
interesting connection to the Arnol$'$d-Pham-Thom conjecture is that
the Tits cone of the $Y_{p,q,r}$ Coxeter group plays a central role in
the treatment of cusp singularities.  (This also requires a small
modification to our basic method, because the Tits cone doesn't come
with a natural metric.)

Our methods apply to many singularities other than the unimodal ones,
but we have confined discussion of these to a few remarks in
section~\ref{sec-singularity-theory}.  An interesting twist that comes
up in the case of ``triangle singularities'' is that the relevant
hyperplane arrangements in $P\Omega$ need not be locally modeled on
those of finite Coxeter groups.  However, it seems likely that our
methods still apply.  See the remark after
lemma~\ref{lem-exceptional-case-lemma}.

The paper is organized as follows.  Section~\ref{sec-conventions}
gives background on \CATx\ geometry.  Sections \ref{sec-cat-0}
and~\ref{sec-positive-curvature-case} give the \CATx\ extension
theorem in the $\x\leq0$ and $\x>0$ cases respectively.  The proof
of the $\x\leq0$ case is quite complicated.  It could be simplified if
one is willing to make more assumptions, such as extendibility of
geodesics, or that distinct geodesics have distinct directions in the
tangent space.  The $\x>0$ case follows quite easily from the
$\x\leq0$ case.

In sections \ref{sec-riem-mflds-local}--\ref{sec-Riemannian-global} we
treat the local and global properties of branched covers of Riemannian
manifolds.  
In section~\ref{sec-Coxeter-arrangements} we discuss our
conjecture about finite
Coxeter groups, and treat the Arnol$'$d-Pham-Thom
conjecture and the asphericity of moduli spaces of amply
lattice-polarized K3s.  Section~\ref{sec-singularity-theory}, on
singularity theory, is the longest and has a different flavor from
the earlier sections.  This is because of the complexity of the
discriminants.  Even Looijenga's elegant
descriptions of them require some work before we can apply our machinery.
We have also given more detail than strictly necessary in hope of
inspiring those who work on Artin groups to look also at the
fundamental groups of these discriminant complements.  These groups
are like Artin groups, but different and perhaps ``better''.  See
\cite{van-der-Lek} for presentations in the simply elliptic and cusp cases.

I would like to thank the Japan Society for the Promotion of Science,
the Clay Mathematics Institute and Kyoto University for their support
and hospitality during part of this work.  I am very grateful to
E.~Looijenga for helping me understand some of the singularity theory
and to K.~Wirthm\"uller for unbinding and scanning his dissertation in
order to send it to me.

\section{Background and Conventions}
\label{sec-conventions}

\noindent
For experts we summarize our nonstandard conventions as
follows.  (1) Geodesics need only be parameterized proportionally to
arclength, not necessarily by arclength.  (2) We use the ``strong''
definition of the \CATx\ inequality, and allow triangles with local
geodesics as edges.  (3)
Lemma~\ref{lem-deformation-of-local-geodesics} slightly refines
standard results on deformations of geodesics.  (4) Our formulation of
the tangent space is equivalent to other treatments but the phrasing
may be new to some.

Let \notation{$(X,d)$} be a metric space. \notation{$B_R(x)$} denotes the open
$R$-ball about $x$.  When we speak of points being \defn{within} some
$\epsilon>0$, we mean that their distance is less than $\epsilon$.

If $\gamma$ is a continuous function from an interval $I=[a,b]$ to $X$,
then we call $\gamma$ a \defn{path} from its \defn{initial endpoint}
$\gamma(a)$ to its \defn{final endpoint} $\gamma(b)$.  Its \defn{length}
\notation{$\ell(\gamma)$} is
$$
\ell(\gamma):=\sup\sum_{i=1}^n
d\bigl(\gamma(t_{i-1}),\gamma(t_i)\bigr)
\in[0,\infty]
$$ where the supremum is over all finite sequences $a=t_0\leq\dots\leq
t_n=b$.  We say $\gamma$ has \defn{speed}${}\leq\sigma$ if the
restriction of $\gamma$ to any subinterval has length at most $\sigma$
times the length of the subinterval.  
$X$ is
called a \defn{length space} and $d$ a \defn{path metric} if the
distance between any two points is the infimum of the lengths of the
paths joining them.

We call $\gamma$ a \defn{geodesic} of speed $\sigma\geq0$ if
$d\bigl(\gamma(s),\gamma(t)\bigr)=\sigma\cdot d(s,t)$ for all $s,t\in I$.  We
call $\gamma$ a \defn{local geodesic} of speed $\sigma$ if this holds
locally on $I$.  We call $\gamma$ a geodesic if it is a geodesic of some
speed, and similarly for local geodesics.  $X$ is  a
\defn{geodesic space} if any two of its points can be joined by a
geodesic.  Often it is convenient to
be sloppy and forget the parameterization, identifying a geodesic or
even a local geodesic with its image.  When there is no
ambiguity about which local geodesic is intended, we will often
indicate it by specifying its endpoints, for example
\notation{$\overline{xy}$}.  A subset $C\sset X$ is called
\defn{convex} if any two of its points may be joined by a geodesic in
$X$ and every geodesic joining them lies in $C$.  

A \defn{triangle} $T$ in $X$ with vertices $x,y,z$ means a choice of
local geo\discretionary{-}{-}{}desics $\overline{xy}$,
$\overline{yz}$, $\overline{zx}$ joining them in pairs.  We call these
the \vardefn{edges}{edge} of $T$, and call $T$ a \defn{geodesic triangle} if
they are geodesics.  Most references discuss only geodesic triangles,
but at times we will be trying to show that a given local geodesic is
actually a geodesic, and the more general formulation will be useful.
When there is no ambiguity about which edges
are intended, we will sometimes specify $T$ by  naming its
vertices.  

Now suppose \nonotation{$\x$}$\x\leq0$ and let
\notation{$X_\x$} be the complete connected simply-connected surface
of constant curvature $\x$---i.e., the Euclidean or hyperbolic plane
if $\x=0$ or $-1$.  A \defn{comparison triangle} $T'$ for $T$ means a
geodesic triangle in $X_\x$ whose edges $\overline{x'y'}$,
$\overline{y'z'}$ and $\overline{z'x'}$ have the same lengths as those
of $T$.  A comparison triangle exists if and only if the lengths of
$T$'s edges satisfy the triangle inequality, and in this case $T'$ is
unique up to isometry of $X_\x$.  In particular, every geodesic
triangle has a comparison triangle.

Suppose that $T$ has a comparison triangle.  Then
to each point of an edge of $T'$, there is a corresponding point on
the corresponding edge of $T$.  (This correspondence is usually
formulated in the other direction; we do it this way since the
edges of $T$ may cross themselves.)  We say that $T$ satisfies
the \defn{\CATx\  inequality} if for any two edges of $T'$ and
points $v'$, $w'$ on them, the inequality
\begin{equation}
\label{eq-CAT-chi-inequality}
d_X(v,w)\leq d_{X_\x}(v',w')
\end{equation}
holds, where $v$ and $w$ are the corresponding points of the
corresponding edges of $T$.  We say $X$ is a \defn{\CATx\ space} if it
is geodesic and every geodesic triangle satisfies this inequality.  We
say $X$ is \defn{locally \CATx}, or has \defn{curvature${}\leq\x$}, if
each point has a neighborhood which is \CATx.  

This definition is sometimes called the ``strong'' form of the
\CATx\ inequality, because some treatments restrict one of $v$, $w$ to
be a vertex of the triangle.  It is well-known that these notions are
equivalent \cite[ch.~3]{Ghys-de-la-Harpe}.  We prefer the strong form
because it admits the Alexandrov subdivision lemma.

The Cartan-Hadamard
theorem for \CATx\ spaces asserts that if $X$ is complete, simply
connected and locally \CATx\ then it is \CATx, hence contractible
\cite[chap. II.4]{bridson-haefliger}.  This is the main reason we care
about \CATx\ spaces.

It would be too much to ask for all triangles to behave like geodesic
triangles.  But Alexandrov's subdivision lemma still holds, with the
same proof.  See \cite[lemma~I.2.16]{bridson-haefliger}.

\begin{lemma}[Alexandrov Subdivision]
\label{lem-Alexandrov-subdivision}
Let $T$ be a triangle in a metric space, with vertices $x$, $y$ and
$z$.  Suppose $w$ is a point of $\overline{yz}$, that $\overline{xw}$
is a local geodesic, and that both triangles $xwy$ and $xwz$ have
comparison triangles and satisfy \CATx.  Then $T$ also has a
comparison triangle and satisfies \CATx.  \qed
\end{lemma}

The following is a slight extension of standard results about
deforming local geodesics.

\begin{lemma}
\label{lem-deformation-of-local-geodesics}
Let $X$ be a metric space, $x,y\in X$, and $\gamma:[0,1]\to X$ a local
geodesic from $x$ to $y$.  Suppose $R>0$ is such that the closed
$R$-ball around each point of $\gamma$ is complete and \CATx.  Then
\begin{enumerate}
\item
\label{item-unique-local-geodesic-uniformly-close}
for all $x_0\in B_{R/2}(x)$ and $y_0\in B_{R/2}(y)$, there exists a
unique local geodesic $\gamma_0$ from $x_0$ to $y_0$ that is
uniformly within $R/2$ of $\gamma$;
\item
\label{item-uniformly-close-local-geodesic-minimizes-length}
$\gamma_0$ minimizes 
length among all paths from $x_0$ to $y_0$ that are uniformly within $R/2$
of $\gamma$;
\item
\label{item-uniformly-close-local-geodesics-are-convex-to-each-other}
if $x_1\in B_{R/2}(x)$ and
$y_1\in B_{R/2}(y)$, then $t\mapsto d(\gamma_0(t),\gamma_1(t))$ is convex.   In particular,
$d(\gamma_0(t),\gamma_1(t))\leq\max\{d(x_0,x_1),d(y_0,y_1)\}$ for all $t$.
\end{enumerate}
Also, suppose $0\leq a\leq b\leq c\leq 1$ and that $T$ is a
triangle whose edges are uniformly within $R/2$ of the restrictions 
of $\gamma$ to $[a,b]$, $[b,c]$ and $[a,c]$.  Then $T$ admits a comparison
triangle and satisfies \CATx.
\end{lemma}

\begin{proof}[Proof sketch.]
The final claim is a consequence of
repeated use of Alex\-androv's lemma.  One should think of $T$ as a ``long
thin triangle'' and subdivide it as
suggested by the following figure:
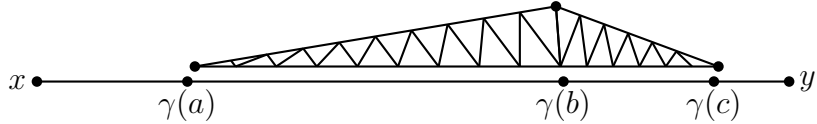
\begin{center}
\psset{unit=.1mm}
\begin{pspicture}(0,-35)(1000,120)
\def\dotsize{7}
\qline(0 ,0)(1000,00)
\rput[r](0,0){$x\,\,$}
\rput[t](200,-10){$\gamma(a)$}
\rput[t](700,-10){$\gamma(b)$}
\rput[t](900,-10){$\gamma(c)$}
\rput[l](1000,0){$\,\,y$}
\qdisk(0,0){\dotsize}
\qdisk(200,0){\dotsize}
\qdisk(700,0){\dotsize}
\qdisk(900,0){\dotsize}
\qdisk(1000,0){\dotsize}
\newcount\tx
\newcount\ty
\newcount\bx
\newcount\by
\def\LEFTx{210}
\def\LEFTy{20}
\qdisk(\LEFTx,\LEFTy){\dotsize}
\def\dx{48}
\def\dy{8}
\def\ex{54}
\tx=\LEFTx
\ty=\LEFTy
\bx=\LEFTx
\by=\LEFTy
\advance\tx by\dx
\advance\ty by\dy
\advance\bx by\ex
\qline(\tx,\ty)(\bx,\by)
\advance\tx by\dx
\advance\ty by\dy
\qline(\tx,\ty)(\bx,\by)
\advance\bx by\ex
\qline(\tx,\ty)(\bx,\by)
\advance\tx by\dx
\advance\ty by\dy
\qline(\tx,\ty)(\bx,\by)
\advance\bx by\ex
\qline(\tx,\ty)(\bx,\by)
\advance\tx by\dx
\advance\ty by\dy
\qline(\tx,\ty)(\bx,\by)
\advance\bx by\ex
\qline(\tx,\ty)(\bx,\by)
\advance\tx by\dx
\advance\ty by\dy
\qline(\tx,\ty)(\bx,\by)
\advance\bx by\ex
\qline(\tx,\ty)(\bx,\by)
\advance\tx by\dx
\advance\ty by\dy
\qline(\tx,\ty)(\bx,\by)
\advance\bx by\ex
\qline(\tx,\ty)(\bx,\by)
\advance\tx by\dx
\advance\ty by\dy
\qline(\tx,\ty)(\bx,\by)
\advance\bx by\ex
\qline(\tx,\ty)(\bx,\by)
\advance\tx by\dx
\advance\ty by\dy
\qline(\tx,\ty)(\bx,\by)
\advance\bx by\ex
\qline(\tx,\ty)(\bx,\by)
\advance\tx by\dx
\advance\ty by\dy
\qline(\tx,\ty)(\bx,\by)
\advance\bx by\ex
\qline(\tx,\ty)(\bx,\by)
\advance\tx by\dx
\advance\ty by\dy
\qline(\tx,\ty)(\bx,\by)
\newcount\topx
\newcount\topy
\topx=\tx
\topy=\ty
\qline(\LEFTx,\LEFTy)(\topx,\topy)
\qdisk(\topx,\topy){\dotsize}
\qline(\tx,\ty)(\bx,\by)
\def\dx{32}
\def\dy{-12}
\def\ex{35}
\advance\tx by\dx
\advance\ty by\dy
\qline(\tx,\ty)(\bx,\by)
\advance\bx by\ex
\qline(\tx,\ty)(\bx,\by)
\advance\tx by\dx
\advance\ty by\dy
\qline(\tx,\ty)(\bx,\by)
\advance\bx by\ex
\qline(\tx,\ty)(\bx,\by)
\advance\tx by\dx
\advance\ty by\dy
\qline(\tx,\ty)(\bx,\by)
\advance\bx by\ex
\qline(\tx,\ty)(\bx,\by)
\advance\tx by\dx
\advance\ty by\dy
\qline(\tx,\ty)(\bx,\by)
\advance\bx by\ex
\qline(\tx,\ty)(\bx,\by)
\advance\tx by\dx
\advance\ty by\dy
\qline(\tx,\ty)(\bx,\by)
\advance\bx by\ex
\qline(\tx,\ty)(\bx,\by)
\advance\bx by\ex
\qline(\LEFTx,\LEFTy)(\bx,\by)
\qline(\topx,\topy)(\bx,\by)
\qdisk(\bx,\by){\dotsize}
\end{pspicture}
\end{center}

The hard part of the lemma is the existence of $\gamma_0$, which is
proven in \cite[lemma II.4.3]{bridson-haefliger} and in \cite[thm.~2]{alexander-bishop-CH-thm}.  Specifically, if $x_0\in B_{R/2}(x)$
and $y_0\in B_{R/2}(y)$ then there is a unique local geodesic
$\gamma_0:[0,1]\to X$ from $x_0$ to $y_0$ for which the function $t\mapsto
d(\gamma(t),\gamma_0(t))$ is convex.  This implies existence in
\eqref{item-unique-local-geodesic-uniformly-close}. 

Now suppose $\beta_0$ and $\beta_1$ are any two local geodesics that are
uniformly within $R/2$ of $\gamma$.  By cutting the quadrilateral with
vertices $\beta_0(0)$, $\beta_0(1)$, $\beta_1(0)$ and $\beta_1(1)$ into two long
thin triangles one can show that $t\mapsto d(\beta_0(t),\beta_1(t))$ is
convex.  The uniqueness part of
\eqref{item-unique-local-geodesic-uniformly-close} follows.  To prove
\eqref{item-uniformly-close-local-geodesics-are-convex-to-each-other}
one just takes $\beta_i=\gamma_i$.  The
length-minimizing property of $\gamma_0$ can also be proven by using long
thin triangles.
\end{proof}

A \defn{geodesic-germ} at $x\in X$ means an equivalence class of
geodesics $[0,\epsilon{>}0]\to X$ with initial endpoint $x$, where two such
are equivalent if they coincide as functions on a neighborhood of $0$.
The constant geodesic at $x$ is allowed.  The following function
\notation{$D$} on pairs of germs is symmetric and satisfies the
triangle inequality:
\begin{equation}
\label{eq-defn-of-distance-in-tangent-space}
D(\gamma,\gamma'):=\limsup_{t\to 0}
\frac{d\bigl(\gamma(t),\gamma'(t)\bigr)}{t}
\in[0,\infty).
\end{equation}
For example, if $\gamma$ and $\gamma'$ differ only by reparameterization, then
their $D$-distance is the difference between their speeds.  If we identify
$\gamma$ and $\gamma'$ when $D(\gamma,\gamma')=0$, then the set of equivalence classes
forms a metric space, called the \defn{tangent space}
\notation{$T_xX$} at $x$.  The basic properties of tangent spaces are
developed in \cite[II3.18--22]{bridson-haefliger}.  That formulation is
slightly different, and uses the term tangent cone rather than tangent
space.  But it is easy to convert between our approach and theirs.

Positive real numbers act on $T_xX$ by scaling the speeds of
geodesic-germs; this scales the metric in the obvious way.  If $X$ is
a Riemannian manifold then $T_xX$ is the usual tangent space, with its
Euclidean metric and the standard scaling.  We give some other
interesting tangent spaces in examples~\ref{eg-cusped-cone}
and~\ref{eg-half-plane-with-crushed-boundary}.

\section{The \texorpdfstring{\CATx}{CAT(chi)} extension theorem}
\label{sec-cat-0}

\noindent
Let $\x\leq0$ be fixed;  see section~\ref{sec-positive-curvature-case}
for the positive-curvature case.

\renewcommand{\theenumi}{\Alph{enumi}}

\begin{theorem}
\label{thm-CATx-extension}
Let \notation{$X$} be a complete geodesic space and \notation{$\D$} a nonempty closed convex
subset.  Assume also:
\begin{enumerate}
\item
\label{hyp-radial-geodesic-triangles-CATx}
every geodesic triangle with a vertex in $\D$ satisfies \CATx;
\item
\label{hyp-tangent-spaces-unique-local-geodesics}
local geodesics in the metric completion $\overline{T_cX}$ of $T_c X$
are unique, for every $c\in \D$; and
\item
\label{hyp-quantified-local-CATx}
there exists \notation{$\lambda$}${}>0$ such that for all $x\in X-\D$, the
  closed ball with center $x$ and radius $\lambda\cdot d(x,\D)$ is
  complete and \CATx.
\end{enumerate}
Then $X$ is \CATx.
\end{theorem}

\renewcommand{\theenumi}{\alph{enumi}}

\begin{example}
\label{eg-cusped-cone}
The non-obvious condition is \eqref{hyp-quantified-local-CATx}, so we
illustrate its role.  Take $X$ to be the surface of revolution of
$y=x^2$, $x\geq0$, around the $x$-axis and $\D$ to be the cusp point.
With $\x=0$, every hypothesis except \eqref{hyp-quantified-local-CATx}
holds, yet $X$ is not \CAT0.  The key point is that the tangent space
at the cusp is a ray, but this forgets too much of the
geometry near the cusp.  An interesting twist on this example is to
take $X'$ equal to the metric completion of the universal cover of
$X-\D$.  The tangent space at the cusp is still just a ray, but now
\eqref{hyp-quantified-local-CATx} holds and $X'$ is \CAT0.
\end{example}

\begin{example}
\label{eg-half-plane-with-crushed-boundary}
This example cautions against weakening hypothesis
\eqref{hyp-radial-geodesic-triangles-CATx}.  Let $X$ be a Euclidean
half-plane with its boundary crushed to a point.  That is,
$X=\{0\}\cup\{(x,y)\in\R^2:x>0\}$, with
\begin{gather*}
d\bigl((x,y),0\bigr)=x\\
d\bigl((x_1,y_1),(x_2,y_2)\bigr)
=
\min\Bigl\{x_1+x_2,\ \sqrt{(x_2-x_1)^2+(y_2-y_1)^2}\Bigr\}
\end{gather*}
and $\D=\{0\}$.  With $\x=0$, every hypothesis holds except
\eqref{hyp-radial-geodesic-triangles-CATx}, which holds in a weaker
form: every triangle with an edge in $\D$ satisfies \CAT0.  Yet $X$ is
not \CAT0.  Here the tangent space at $0$ is the cone on an uncountable
discrete set.  In a sense this example is the opposite of the previous
one, because the metric topology on the space of geodesics emanating
from $0$ is much finer than the topology of uniform convergence,
whereas it was much coarser in example~\ref{eg-cusped-cone}.
\end{example}

The following lemma contains almost all the content of the theorem.

\begin{lemma}
\label{lem-unique-local-geodesics}
There is a unique local geodesic joining any two points of $X$.
\end{lemma}

\begin{proof}[Proof of theorem~\ref{thm-CATx-extension}, given lemma~\ref{lem-unique-local-geodesics}:]
We must show that every geodesic triangle $T$ satisfies the
\CATx\ inequality.  If $T$ has a vertex in $\D$ then this is hypothesis
(\ref{hyp-radial-geodesic-triangles-CATx}).  If an edge of $T$ meets
$\D$ then we use Alexandrov subdivision and reduce to the previous case.  If
$\overline{xy}$ meets $\D$ for some vertex $x$ of $T$ and point $y$ of
the opposite edge $E$, then we subdivide along $\overline{xy}$ and
reduce to the previous case.  In the remaining case we consider the
set of geodesics $\overline{xy}$ from $x$ to the points $y$ of $E$.
None of them meets $\D$, and therefore (lemma~\ref{lem-deformation-of-local-geodesics}) each may be varied through a
continuous family of local geodesics from $x$ to nearby points of $E$.
By lemma~\ref{lem-unique-local-geodesics}, these local geodesics are
the same as the geodesics we started with.  That is, the geodesics
$\overline{xy}$ vary continuously with $y$.  Now the Alexandrov
patchwork argument \cite[fig.~4.2]{bridson-haefliger} proves that $T$
satisfies \CATx.
\end{proof}

The rest of the section is devoted to proving lemma~\ref{lem-unique-local-geodesics}.  Our first
step is to improve hypothesis
\eqref{hyp-radial-geodesic-triangles-CATx} to apply to all triangles
with a vertex in $\D$, not just geodesic ones.  This also proves a
special case of lemma~\ref{lem-unique-local-geodesics}.

\begin{lemma}
\label{lem-radial-triangles-are-CATx}
Every triangle with a vertex in $\D$ admits a comparison triangle and
satisfies \CATx.  Also, there is a unique local geodesic from any
given point of $\D$ to any given point of $X$.
\end{lemma}

\begin{proof}
Suppose $c\in \D$ and $x,y\in X$.  First we prove a special case of the
first claim.  If $cxy$ is a triangle such that $\overline{cx}$ and
$\overline{cy}$ are geodesics, then one can subdivide $\overline{xy}$
into geodesic segments and join the subdivision points to $c$ by geodesics.  Each triangle
obtained by subdivision satisfies \CATx\ by hypothesis~\eqref{hyp-radial-geodesic-triangles-CATx}, and
repeated use of Alexandrov's lemma shows that $cxy$ has a comparison
triangle and satisfies \CATx.  

Now we prove the second claim.  We must
show that if $\gamma$ is a local geodesic from $c$ to $x$ then it equals
$\overline{cx}$.  We use the case just proven with $y=c$ and
$\overline{cy}$ being the constant path at $y$.  The comparison
triangle is then a segment, and the \CATx\ inequality forces
$\gamma=\overline{cx}$.  Finally, if $cxy$ is any triangle with a vertex
at $c$, then we have just seen that $\overline{cx}$ and
$\overline{cy}$ are geodesics, and the lemma reduces to the special
case.
\end{proof}

We will use lemma~\ref{lem-radial-triangles-are-CATx} many times, without specific reference.  What
remains to prove lemma~\ref{lem-unique-local-geodesics} is the case
$x,y\in X-\D$.  Before the real work begins there is one more easy
case: when $x$ and $y$ are joined by a local geodesic long enough to
touch $\D$.  Define \notation{$B(x,y)$}${}=\inf_{c\in \D}\bigl(d(x,c)+d(c,y)\bigr)$,
and call any $c$ realizing this infimum a \defn{center} for $x$ and
$y$.  Standard arguments \cite[prop.~II.2.7]{bridson-haefliger} show
that $c$ exists.  (It is also unique, but we won't need this.)

\begin{lemma}
\label{lem-far-apart-case-new-version}
If $x$ and $y$ are \defn{far apart}, meaning that they are joined by a
local geodesic $\beta$ of length${}\geq B(x,y)$, then that is the unique
local geodesic joining them.
\end{lemma}

\begin{proof}
Applying the \CATx\ inequality to the triangle $\overline{xc},
\overline{cy},\beta$ shows that $\beta$ equals $\overline{xc}\cup\overline{cy}$.  This
argument applies to any local geodesic of length${}\geq B(x,y)$, so
$\beta$ is the unique local geodesic of length${}\geq B(x,y)$.  The
\CATx\ inequality also implies that there cannot be a local geodesic
$\beta'$ from $x$ to $y$ of length${}<B(x,y)$; otherwise
$\overline{xc}\cup\overline{cy}$ would fail to be a local geodesic at
$c$
\end{proof}

For the rest of the section we suppose $x,y$ are not far apart and
that \notation{$c$}${}\in \D$ is their center.  The proof of
lemma~\ref{lem-unique-local-geodesics} in this case relies on several
delicate rescaling arguments.

For $z\in X$
and $t\in[0,1]$ we define $t.z$ to be the unique point of
$\overline{cz}$ at distance $t\cdot d(c,z)$ from $c$.  For $w,z\in X$
and $t\in(0,1]$ we define the possibly-degenerate metric
$d_t(w,z)=\frac{1}{t}d(t.w,t.z)$.  
The \CATx\ inequality for triangles with a vertex at $c$ implies
that for fixed $w$ and $z$, $d_t(w,z)$ is nonincreasing as $t\to0$, so
it has a limit $d_0(w,z)$.  In fact this limit is
$D(\overline{cw},\overline{cz})$, the distance in $T_cX$ between
$\overline{cw}$ and $\overline{cz}$, where both geodesics are
parameterized by $[0,1]$.  So $d_0$ describes the geometry of
$T_c X$.

Now suppose \notation{$\beta$} is a local geodesic from $x$ to $y$; ultimately we
will show it is the only one.  Since $x$ and $y$ are not far apart,
$\beta$ is too short to meet $\D$, so it lies in $X-\D$, which is locally
complete and locally \CATx.  Therefore deforming $\beta$'s endpoints
along $\overline{cx}$ and $\overline{cy}$ yields a deformation of $\beta$
through local geodesics in $X-\D$.  Our first result is that this
deformation meets no obstructions until the endpoints reach $c$.

\begin{lemma}
\label{lem-no-obstructions-to-deformation}
There is a unique continuous map $(0,1]\times[0,1]\to X-\D$, which we
  write $(t,s)\mapsto{}$\notation{$\beta_t(s)$}, with $\beta_1=\beta$ and $\beta_t$ a local
  geodesic from $t.x$ to $t.y$.  Furthermore, 
\begin{enumerate}
\item
\label{item-beta-length-shrinks-at-least-linearly}
$\frac{1}{t}\ell(\beta_t)$
  is nonincreasing as $t$ decreases; in particular $\ell(\beta_t)\leq
  t\ell(\beta)$.
\item
\label{item-beta-distance-bound-from-Delta}
Every point of $\beta_t$ lies at
distance${}\geq\frac{t}{2}\bigl(B(x,y)-\ell(\beta)\bigr)>0$ from $\D$.
\end{enumerate}
\end{lemma}

\begin{proof}
Certainly there is some $0\leq t_0<1$ for which the first assertion
holds with $(t_0,1]$ in place of $(0,1]$.  We will prove
    \eqref{item-beta-length-shrinks-at-least-linearly} and
    \eqref{item-beta-distance-bound-from-Delta} for $t\in(t_0,1]$,
      then use them to show that we may take $t_0=0$.
      
\eqref{item-beta-length-shrinks-at-least-linearly} is
essentially a standard fact, but since the triangle $cxy$ is not in a
(known) \CATx\ space and the $\beta_t$'s are not (known) geodesics, we
indicate a proof.  
Given $t\in(t_0,1]$, suppose $t'$ is slightly smaller than
$t$.  By moving one endpoint of $\beta_t$, we obtain a local geodesic from $t'\!.x$ to
$t.y$, and then by moving the other we obtain $\beta_{t'}$.  This uses
  the uniqueness in lemma~\ref{lem-deformation-of-local-geodesics}.  The triangle
with vertices $c$, $t'\!.x$ and $t'\!.y$ satisfies \CATx\ by lemma~\ref{lem-radial-triangles-are-CATx}.  The triangle with
vertices $t.x$, $t'\!.x$ and $t.y$ satisfies \CATx\ by
lemma~\ref{lem-deformation-of-local-geodesics}, and similarly
for the triangle with vertices $t'\!.x$, $t.y$ and $t'\!.y$.  Assemble
the comparison triangles in the obvious way.  Then the interior angle
at the point corresponding to $t'\!.x$ must be at least $\pi$, or else
$\overline{cx}$ would fail to be a geodesic at $t'\!.x$.  And similarly for
$t'\!.y$.  This implies $\ell(\beta_{t'})\leq\frac{t'}{t}\ell(\beta_t)$.

\eqref{item-beta-distance-bound-from-Delta} For $t\in(t_0,1]$,
  consider the path along $\overline{cx}$ from $x$ to $t.x$, then
  along $\beta_t$, and then along $\overline{cy}$ from $t.y$ to~$y$.
  Using \eqref{item-beta-length-shrinks-at-least-linearly} and
  $B(x,y)=d(x,c)+d(c,y)$, its length is bounded above by
\begin{equation}
\label{eq-beta-t-is-far-away-from-branch-locus}
(1-t)d(x,c)+t\ell(\beta)+(1-t)d(y,c)
=B(x,y)-t\bigl(B(x,y)-\ell(\beta)\bigr).
\end{equation}
Since every path from $x$ to $y$ that meets $\D$ has length${}\geq
B(x,y)$, every point of $\beta_t$ lies at distance at least
$\frac{t}{2}\bigl(B(x,y)-\ell(\beta)\bigr)$ from $\D$, as desired.
Also, $B(x,y)-\ell(\beta)>0$ since $x$ and $y$ are not far apart.

Now we show that we may take $t_0=0$; suppose $t_0>0$.  The $\beta_t$ are
uniformly Cauchy (lemma~\ref{lem-deformation-of-local-geodesics}), so
they have a limit $\beta_{t_0}$ in $X$.  Since
\eqref{eq-beta-t-is-far-away-from-branch-locus} holds for all
$\beta_{t>t_0}$, it also holds for $\beta_{t_0}$.  That is, $\beta_{t_0}$ lies
at distance at least $\frac{t_0}{2}\bigl(B(x,y)-\ell(\beta)\bigr)>0$ from
$\D$, so it is covered by \CATx\ neighborhoods.  In each of these
$\beta_{t_0}$ is a limit of geodesics, so $\beta_{t_0}$ is a local geodesic.
So we can continue the deformation by deforming $\beta_{t_0}$.  It
follows that we may take $t_0=0$.
\end{proof}

Now we define scaled-up versions of the $\beta_t$.  Because we haven't
assumed any sort of extendibility of geodesics, we must do the scaling
in $T_cX$.  There is a natural projection \defn{$\pi$}${}:X\to T_c X$,
defined by assigning each $z\in X$ to the germ of the geodesic
$[0,1]\to X$ from $c$ to $z$.  (Caution: $\pi$ may be
discontinuous, such as in
example~\ref{eg-half-plane-with-crushed-boundary}.)  For $t\in(0,1]$ we
  define \notation{$\gamma_t$}${}:[0,1]\to T_c X$ to be $\beta_t$, followed
  by $\pi$, followed by scaling up by $1/t$.  
It follows from \eqref{hyp-radial-geodesic-triangles-CATx} that the $\gamma_t$ are continuous, but we won't
actually use this.

Our next goal is lemma~\ref{lem-gamma-t-convereges-uniformly}, which
yields a limit $\gamma_0$ of the $\gamma_t$'s.  We will need the following
tool for showing that a family of curves are uniformly close when one
of them is a local geodesic and the others are not much faster than
it. This lemma is independent of the hypotheses of
theorem~\ref{thm-CATx-extension}.

\begin{lemma}
\label{lem-not-much-longer-implies-uniformly-close}
Let $L\geq0$, $R>0$ and $\epsilon>0$ be given.  Then there exists $\delta>0$
such that the following holds.  Suppose $X$ is a metric space and
$\alpha:[0,1]\to X$ is a local geodesic of length${}\leq L$, such that the
closed $R$-ball around every point of $\alpha$ is complete and \CAT0.  If
$\alpha_u$ is a continuous variation of $\alpha$ through paths
$[0,1]\to X$ of
speed${}\leq\ell(\alpha)+\delta$, with the same endpoints, then each $\alpha_u$ is
uniformly within $\epsilon$ of $\alpha$.

Furthermore, if one scales $L$, $R$ and $\epsilon$ by a positive number then
one may scale $\delta$ by the same factor.
\end{lemma}

\begin{proof}
Suppose $L$, $R$ and $\epsilon$ are given and without loss of generality assume
$\epsilon<R/2$.  It is easy to prove an analogue of the lemma for $X=\R^2$:
there exists $\delta>0$ such that for any geodesic
$\alpha':[0,1]\to\R^2$ of length${}\leq L$ and any path $\alpha_u'$ of
speed${}\leq\ell(\alpha')+\delta$ having the same endpoints, $\alpha_u'$ and $\alpha'$
are uniformly within $\epsilon$.  The only care required is that one must
know $L$ before choosing $\delta$.  We will prove that this $\delta$ satisfies the
lemma.

So suppose $X$, $\alpha$, $\alpha_u$ are as stated, with $u$ varying over
$[0,1]$ and $\alpha_0=\alpha$.  Write $a$ and $c$ for the common endpoints of all
these paths.  Let $I$ be the set of $u$ for which $\alpha_u$ is uniformly
within $\epsilon$ of $\alpha$.  This is open, so by connectedness it will
suffice to show that $[0,u)\sset I$ implies $u\in I$.  So let
  $s\in[0,1]$ and write $b$ for $\alpha_u(s)$.  We must show
  $d(b,\alpha(s))<\epsilon$.  

Now, $\alpha_u$ is uniformly at most $\epsilon<R/2$ from $\alpha$, since it is a
limit of paths with this property.  Let $\overline{ab}$ (resp.\ $\overline{bc}$) be the unique geodesic from $a$ to $b$ (resp.\ $b$ to
$c$) that is uniformly within $R/2$ of $\alpha|_{[0,s]}$
(resp.\ $\alpha|_{[s,1]}$).  These exist by
lemma~\ref{lem-deformation-of-local-geodesics}, which also
tells us that $\overline{ab}$ and $\overline{bc}$ are no longer than
the corresponding parts of $\alpha_u$, and that the triangle
$\overline{ab}$, $\overline{bc}$, $\alpha$ satisfies \CAT0.  Let
$a',b',c'$ be the vertices of the comparison triangle.  Note that
$$
d(a',b')
=
\ell\bigl(\overline{ab}\bigr)
\leq
\ell(\alpha_u|_{[0,s]})
\leq 
s(\ell(\alpha)+\delta)
=
s(\ell(\alpha')+\delta)
$$
and similarly $d(b',c')\leq(1-s)(\ell(\alpha')+\delta)$.  Therefore there is a
path $\alpha_u':[0,1]\to\R^2$ from $a'$ to $c'$ of
speed${}\leq\ell(\alpha')+\delta$, with $\alpha_u'(s)=b'$.  Then we have
$$
d(b,\alpha(s))
\leq
d(b',\alpha'(s))
<
\epsilon,
$$
the first step by the \CAT0 inequality and the second by the choice of
$\delta$.  

We have proven that $\delta$ has the property in the statement of the
lemma.  The final claim of the lemma follows by scaling all distances.
\end{proof}

The conceptual content of the next lemma is its conclusion
\eqref{item-gammas-uniformly-Cauchy}, that the $\gamma_t$ have a limit
$\gamma_0$.  But the technical content is its conclusion
\eqref{item-betas-uniformly-Cauchy-after-rescaling}, that the $\beta_t$
are ``uniformly Cauchy after rescaling''.  Several constants appear in
the proof that we will need again later, so we define them here.  We
set \notation{$L$}${}=\ell(\beta)$ and
\notation{$L_0$}${}=\lim_{t\to0}\frac{1}{t}\ell(\beta_t)$, which exists
by
lemma~\ref{lem-no-obstructions-to-deformation}\eqref{item-beta-length-shrinks-at-least-linearly}.
Also, we define \notation{$R$}${}=\frac{\lambda}{2}\bigl(B(x,y)-\ell(\beta)\bigr)$,
where $\lambda$ is from hypothesis \eqref{hyp-quantified-local-CATx}.
The point of this definition is that any point of $\beta_t$ lies at
distance${}\geq\frac{t}{2}\bigl(B(x,y)-\ell(\beta)\bigr)$ from $\D$
(lemma~\ref{lem-no-obstructions-to-deformation}\eqref{item-beta-distance-bound-from-Delta}),
so the closed $tR$-ball centered there is complete and \CATx\ by hypothesis
\eqref{hyp-quantified-local-CATx}.

\begin{lemma}
\label{lem-gamma-t-convereges-uniformly}
\leavevmode
\begin{enumerate}
\item
\label{item-betas-uniformly-Cauchy-after-rescaling}
If $\epsilon>0$ is given, then for all small enough
$t>0$ and all $\mu\in(0,1]$, $\mu.\beta_t$ and $\beta_{\mu t}$ are uniformly
within $\mu\epsilon t$.  
\item
\label{item-gammas-uniformly-Cauchy}
The functions $\gamma_t$ converge uniformly to a function
\notation{$\gamma_0$}${}:[0,1]\to\overline{T_cX}$.
\end{enumerate}
\end{lemma}

\begin{proof}
\eqref{item-betas-uniformly-Cauchy-after-rescaling}
Suppose $\epsilon>0$ is
given.  Take $\delta$ from lemma~\ref{lem-not-much-longer-implies-uniformly-close}, using the values $L$ and $R$
given above.  Now suppose $t$ is small enough that
\begin{equation}
\label{eq-lengths-close-to-limiting-length}
L_0\leq\frac{1}{t}\ell(\beta_t)<L_0+\delta.
\end{equation}
We will prove that for any fixed $\mu\in(0,1]$, $\mu.\beta_t$ is
  uniformly within $\epsilon\mu t$ of $\beta_{\mu t}$.  To do this we will
  apply lemma~\ref{lem-not-much-longer-implies-uniformly-close} to the family
  of paths $\nu.\beta_{\mu t/\nu}$, $\nu\in[\mu,1]$, which interpolate
  between $\mu.\beta_t$ ($\nu=\mu$) and $\beta_{\mu t}$ ($\nu=1$).  We
  regard this as a deformation of $\beta_{\mu t}$, and we will apply the
  ``rescaled'' version of
  lemma~\ref{lem-not-much-longer-implies-uniformly-close} (i.e., its last
  assertion).

To apply that lemma we  verify (i) $\ell(\beta_{\mu t})\leq\mu tL$,
(ii) the closed $\mu tR$-ball around every point of $\beta_{\mu t}$ is
complete and \CAT0, and (iii) every path $\nu.\beta_{\mu t/\nu}$ has
speed${}\leq\ell(\beta_{\mu t})+\delta\mu t$.  The first condition holds
because $\ell(\beta_{\mu t})\leq\mu t\ell(\beta)=\mu tL$ by
lemma~\ref{lem-no-obstructions-to-deformation}\eqref{item-beta-length-shrinks-at-least-linearly}.
The second condition holds by choice of $R$, as explained above.  To address the third
condition, recall that $\beta_{\mu t/\nu}$ is a local geodesic
parameterized by $[0,1]$.  Therefore its speed is
$\ell(\beta_{\mu t/\nu})$.  Then the \CATx\ inequality implies that
$\nu.\beta_{\mu t/\nu}$ has speed at most
$$
\nu\ell(\beta_{\mu t/\nu})
\leq
\nu\frac{\mu}{\nu}\ell(\beta_t)
=
\mu t\frac{\ell(\beta_t)}{t}
<
\mu t\Bigl(\frac{\ell(\beta_{\mu t})}{\mu t}+\delta\Bigr)
=
\ell(\beta_{\mu t})+\delta\mu t,
$$
as desired.  The first step uses lemma~\ref{lem-no-obstructions-to-deformation}\eqref{item-beta-length-shrinks-at-least-linearly} and the third
uses \eqref{eq-lengths-close-to-limiting-length}.   We have verified the 
hypotheses of lemma~\ref{lem-not-much-longer-implies-uniformly-close}, so we deduce that all members of the
deformation, in particularly $\mu.\beta_t$, are uniformly within $\mu
t\epsilon$ of $\beta_{\mu t}$.

\eqref{item-gammas-uniformly-Cauchy} We will prove the $\gamma_t$
uniformly Cauchy as $t\to0$.  Suppose $\epsilon>0$ is given, that $t$ is
small enough to satisfy \eqref{eq-lengths-close-to-limiting-length},
and $\mu\in(0,1]$ and $s\in[0,1]$.  Then
\begin{align*}
D\bigl(\gamma_t(s),\gamma_{\mu t}(s)\bigr)
&=
D\bigl({\textstyle\frac{1}{t}}\,\pi\circ\beta_t(s),
{\textstyle\frac{1}{\mu t}}\,\pi\circ\beta_{\mu t}(s)\bigr)
\displaybreak[0]
\\
&=
D\bigl({\textstyle\frac{1}{\mu t}}\,\pi(\mu.\beta_t(s)),
{\textstyle\frac{1}{\mu t}}\,\pi\circ\beta_{\mu t}(s)\bigr)
\displaybreak[0]
\\
&=
{\textstyle\frac{1}{\mu t}}D\bigl(\pi(\mu.\beta_t(s)),\pi\circ\beta_{\mu t}(s)\bigr)
\displaybreak[0]
\\
&=
{\textstyle\frac{1}{\mu t}}
d_0\bigl(\mu.\beta_t(s),\beta_{\mu t}(s)\bigr)
\displaybreak[0]
\\
&\leq
{\textstyle\frac{1}{\mu t}}
d\bigl(\mu.\beta_t(s),\beta_{\mu t}(s)\bigr)
<
{\textstyle\frac{1}{\mu t}}
\mu t\epsilon
=\epsilon.
\end{align*}
\end{proof}

Our next goal is that $\gamma_0$ is a local geodesic.  Note that the
$\gamma_{t>0}$ are usually not geodesics, as simple examples show.

\begin{lemma}
\label{lem-gamma-0-is-local-geodesic}
$\gamma_0$ is a local geodesic.
\end{lemma}

\begin{proof}
If $0\leq s<s'\leq1$ are within $R/L$ of each other then we call
$\beta_t|_{[s,s']}$ a short segment of $\beta_t$.  Here $R$ and $L$ are as
before lemma~\ref{lem-gamma-t-convereges-uniformly}, and we assume
$L\neq0$ because the $L=0$ case of the lemma is
trivial.  The importance of short segments is that the speed of $\beta_t$
is at most $tL$, so a short segment of $\beta_t$ lies in the $tR$-ball
around each of its points, which is \CATx\ by hypothesis.  Since
$\beta_t$ is a local geodesic, its short segments are geodesics.
We will show that each short segment of $\gamma_0$ is a geodesic, which
proves the lemma.

So suppose $s'-s<R/L$.  For $\epsilon>0$, choose $t>0$ small enough that
$\gamma_{t}$ is uniformly within $\epsilon$ of $\gamma_0$.  By shrinking $t$, we may
also suppose that the conclusion of
lemma~\ref{lem-gamma-t-convereges-uniformly}\eqref{item-betas-uniformly-Cauchy-after-rescaling}
holds.  In the following calculation we allow $\mu$ but not $t$ to
vary.  Given $s=s_0<\dots<s_m=s'$, we first use the convergence
$\gamma_{\mu t}\to\gamma_0$:
\begin{align*}
\sum_{i=1}^m D\bigl( \gamma_0(s_{i-1}), \gamma_0(s_i) \bigr)
&{}=
\lim_{\mu\to0}\sum_{i=1}^m D\bigl(\gamma_{\mu t}(s_{i-1}),\gamma_{\mu
  t}(s_i)\bigr)
\\
\noalign{\noindent then rescaling in $T_cX$:}
&{}=\lim_{\mu\to0}\frac{1}{\mu t}\sum_{i=1}^m d_0\bigl(\beta_{\mu
  t}(s_{i-1}),\beta_{\mu t}(s_i)\bigr)
\\
\noalign{\noindent then $d_0\leq d$:}
&{}\leq
\lim_{\mu\to0}\frac{1}{\mu t}\sum_{i=1}^m d\bigl(\beta_{\mu
  t}(s_{i-1}),\beta_{\mu t}(s_i)\bigr)
\\
\noalign{\noindent then the fact that $\beta_{\mu t}|_{[s,s']}$ is a geodesic:}
&{}=
\lim_{\mu\to 0}\frac{1}{\mu t}d\bigl(\beta_{\mu t}(s),\beta_{\mu
  t}(s')\bigr)
\\
\noalign{\noindent
and then lemma~\ref{lem-gamma-t-convereges-uniformly}\eqref{item-betas-uniformly-Cauchy-after-rescaling}:}
&{}\leq
\lim_{\mu\to0}\frac{1}{\mu t}
\Bigl(d\bigl(\mu.\beta_{t}(s),\mu.\beta_{t}(s')\bigr)+2\mu t\epsilon\Bigr)
\\
&{}=
2\epsilon+\frac{1}{t} d_0\bigl(\beta_{t}(s),\beta_{t}(s')\bigr)
\\
&{}=
2\epsilon+D\bigl(\gamma_{t}(s),\gamma_{t}(s')\bigr)
\\
&{}<
4\epsilon + D\bigl(\gamma_0(s),\gamma_0(s')\bigr).
\end{align*}  
Now, $\ell_D\bigl(\gamma_0|_{[s,s']}\bigr)$ is the supremum of the left
hand side over all choices of $s_0,\dots,s_m$, so it is bounded above
by the right hand side. Since this holds for all
$\epsilon>0$, $\ell_D\bigl(\gamma_0|_{[s,s']}\bigl)$ is bounded above by
$D\bigl(\gamma_0(s),\gamma_0(s')\bigr)$.  

This shows that $\gamma_0$ is a local geodesic except perhaps for its
parameterization.  Since
$\ell_D\bigl(\gamma_0|_{[s,s']}\bigr)=D\bigl(\gamma_0(s),\gamma_0(s')\bigr)$ for
all short intervals $[s,s']$, it follows that $\gamma_0$ has constant
speed, so it is a local geodesic.
\end{proof}

\begin{lemma}
\label{lem-beta-and-beta-prime-become-closer-faster-than-linearly-as-they-shrink}
Suppose $\beta$ and $\beta'$ are local geodesics from $x$ to $y$.  Given
$\epsilon>0$, there exist $t,\mu\in(0,1]$ such that $\beta_{\mu t}$ and
$\beta_{\mu t}'$ are uniformly within $4\mu t\epsilon$.
\end{lemma}

\begin{proof}
Applying the constructions beginning with
lemma~\ref{lem-no-obstructions-to-deformation} to $\beta'$ as we did to $\beta$,
we obtain another local geodesic $\gamma_0'\sset\overline{T_cX}$ from
$\pi(x)$ to $\pi(y)$.  By hypothesis
\eqref{hyp-tangent-spaces-unique-local-geodesics}, this coincides with
$\gamma_0$.  Now suppose $\epsilon>0$.  First we choose $t>0$ small enough that
$\gamma_t$ and $\gamma_t'$ are uniformly within $\epsilon$, and such that the
conclusion of
lemma~\ref{lem-gamma-t-convereges-uniformly}\eqref{item-betas-uniformly-Cauchy-after-rescaling}
holds for both $\beta$ and $\beta'$.  Then we choose $\mu$ small enough that
$d_\mu$ is uniformly within $t\epsilon$ of $d_0$ on $\beta_t\cup\beta_t'$.  (The
uniform convergence $d_\mu\to d_0$ on compact sets in elementary.)
Now we suppose $s\in[0,1]$ and apply
lemma~\ref{lem-gamma-t-convereges-uniformly}\eqref{item-betas-uniformly-Cauchy-after-rescaling}:
\begin{align*}
d\bigl(\beta_{\mu t}(s),\beta_{\mu t}'(s)\bigr)
&{}<
2t\mu\epsilon + d\bigl(\mu.\beta_t(s),\mu.\beta_t'(s)\bigr)
\\
\noalign{\noindent then the definition of $d_\mu$:} 
&{}=
2t\mu\epsilon+\mu d_\mu\bigl(\beta_t(s),\beta_t'(s)\bigr)
\\
\noalign{\noindent and then $d_\mu\approx d_0$:}
&{}<
2t\mu\epsilon +\mu\Bigl(t\epsilon+d_0\bigl(\beta_t(s),\beta_t'(s')\bigr)\Bigr)
\\
&{}=
3t\mu\epsilon+\mu D\bigl(\pi\circ\beta_t(s),\pi\circ\beta_t'(s)\bigr)
\\
&{}=
3t\mu\epsilon + \mu t D\bigl(\gamma_t(s),\gamma_t'(s)\bigr)
\\
&{}<
4t\mu\epsilon.
\end{align*}
\end{proof}

\begin{proof}[Conclusion of the proof of lemma~\ref{lem-unique-local-geodesics}:]
Suppose $x,y\in X-\D$ are not far apart and that $\beta,\beta'$ are local
geodesics from $x$ to $y$.  We must show $\beta=\beta'$.  We apply
lemma~\ref{lem-beta-and-beta-prime-become-closer-faster-than-linearly-as-they-shrink}
with $\epsilon=R/8$, so $\beta_{\mu t}$ and $\beta_{\mu t}'$ are uniformly within
$\mu tR/2$, for some $\mu,t\in(0,1]$.  By our choice of $R$, the closed $\mu
  tR$-ball around each point of $\beta_{\mu t}$ is \CATx.  So
  lemma~\ref{lem-deformation-of-local-geodesics} says that there
  is a unique local geodesic from $\mu t.x$ to $\mu t.y$ that is
  uniformly within $\mu t R/2$ of $\beta_{\mu t}$.  This implies $\beta_{\mu
    t}'=\beta_{\mu t}$.

We obtained $\beta_{\mu t}$ from $\beta$ by deforming $\beta$ through
a family of local geodesics in $X-\D$, and we can reverse this
deformation to recover $\beta$ from $\beta_{\mu t}$.  And similarly for
$\beta'$.  Since the deformation is unique given the
motion of the endpoints, $\beta'_{\mu t}=\beta_{\mu t}$ implies $\beta'=\beta$, finishing
the proof.
\end{proof}

\section{The Positive-Curvature Case}
\label{sec-positive-curvature-case}

\noindent
In this section, we give the positive-curvature analogue of
theorem~\ref{thm-CATx-extension}.  Happily, it follows from the
non-positive-curvature case.

If $\x>0$ then $X_\x$ is the sphere of radius $1/\sqrt\x$.  If $T$ is
a geodesic triangle in a metric space, then it need not have a
comparison triangle $T'$ in $X_\x$, and if it does then $T'$ need not
be unique.  However, if $T$ has perimeter${}<\circum
X_\x=2\pi/\sqrt\x$, then $T'$ exists and is unique.  In this case we
say that $T$ satisfies \CATx\ if \eqref{eq-CAT-chi-inequality} holds.
We call a geodesic space \CATx\ if every geodesic triangle of
perimeter${}<\circum X_\x$ satisfies \CATx.

\begin{theorem}
\label{thm-positive-curvature-case}
Suppose $\x>0$ and that $X$ and $\D$ are as in theorem~\ref{thm-CATx-extension}.  Assume
also that all of $X$ lies within $R<\frac{1}{4}\circum X_\x$ of some
fixed point of $\D$.  Then $X$ is \CATx.
\end{theorem}

\begin{proof}
By scaling, it suffices to treat the case $\x=1$ (in which case
$R<\pi/2$).  By the positive-curvature form of the Cartan-Hadamard
theorem \cite[thm. 4.3]{alexander-bishop-curved-curves}, it suffices
to show that $X$ is locally \CAT1.  This uses our hypothesis about $X$
lying in the $R$-ball around a point of $\D$.  (We need a point in
$\D$, not just $X$, so that our
hypothesis~\eqref{hyp-radial-geodesic-triangles-CATx} implies the
radial uniqueness hypothesis of
\cite{alexander-bishop-curved-curves}.)

We can convert
this into a \CAT0 problem by defining $CX$ as the Euclidean cone on
$X$, with vertex say $v$, and $C\D\sset CX$ to be the cone on $\D$.
We regard $X$ as a subset of $CX$, namely the unit sphere around $v$.
The given metric on $X$ may be recovered as the path metric induced on
it by the restriction of $CX$'s metric.

The basic properties of Euclidean cones appear in \cite{bridson-haefliger}.
In particular, \cite[theorem~I.5.10]{bridson-haefliger} states that a geodesic in
$CX$ between $t_1.x_1$ and $t_2.x_2$ ($t_1,t_2>0$, $x_1,x_2\in X$)
misses $v$ and may be projected radially to $X$, yielding a geodesic in $X$.  (This
uses $\diam X<\pi$.)  This establishes a bijection between the
geodesics of $CX$ from $t_1.x_1$ to $t_2.x_2$ and the geodesics
of $X$ from $x_1$ to $x_2$.    Also, suppose $T$ is a geodesic
triangle in $CX-\{v\}$ whose radial projection to $X$ has
perimeter${}<2\pi$.  Then $T$ satisfies \CAT0\ if and only if its
radial projection
satisfies \CAT1.  (See the proof of \cite[thm II.3.18]{bridson-haefliger}.)  Therefore,
proving $X$ locally \CAT1\ is equivalent to proving $CX-\{v\}$
locally \CAT0.

This lets us apply theorem~\ref{thm-CATx-extension}.  Clearly $CX$ is locally \CAT0
away from $C\D$.  Now, an element of $C\D-\{v\}$ has the form $t.c$ with
$t>0$ and $c\in \D$.  We  choose a small closed ball $X_0$ around
$t.c$ in such a way that theorem~\ref{thm-CATx-extension} applies to $X_0$ and
$\D_0:=X_0\cap C\D$.  Indeed, a closed ball of any radius${}<t$ will
do; we now check the hypotheses of theorem~\ref{thm-CATx-extension}.  The completeness
of $X_0$ uses the completeness of $CX$ (see
\cite[prop.~I.5.9]{bridson-haefliger}).  That $X_0$ is a geodesic space and $\D_0$ is
convex follow from the correspondences between geodesics in $CX$
and in $X$.  Hypotheses \eqref{hyp-radial-geodesic-triangles-CATx} and \eqref{hyp-quantified-local-CATx} follow immediately from
the corresponding hypotheses on $X$.  Hypothesis \eqref{hyp-tangent-spaces-unique-local-geodesics} follows from
the corresponding hypothesis on $X$, together with the observation
that each tangent space to $C X-\{v\}$ is $\R$ times a tangent space
to $X$.    So $X_0$ is \CAT0.
\end{proof}

\section{Branched Covers of Riemannian Manifolds: Local Properties}
\label{sec-riem-mflds-local}

\noindent
Suppose \notation{$M$} is a complete Riemannian manifold with
sectional cur\discretionary{-}{-}{}vature${}\leq\x\in\R$ and
\nonotation{$\D$}$\D\sset M$ is locally the union of finitely many
complete totally geodesic submanifolds of codimension~2.  We equip any
covering space \notation{$M_0'$} of \nonotation{$M_0$}$M_0:=M-\D$,
with its natural path metric and complete it to obtain a metric space
\notation{$M'$}.  We call \nonotation{$\pi$}$\pi:M'\to M$ the
\defn{branched cover} associated to $M_0'\to M_0$.  When $M$ is
connected and $M_0'$ is the universal cover of $M_0$ then we call $M'$
the \defn{universal branched cover} of $M$ over $\D$.  Here is the
main result of this section:

\begin{theorem}
\label{thm-riem-mfld-br-cover-locally-CATx}
$M'$ has curvature${}\leq\x$ if and
only if each tangent space $T_{x'}M'$ is \CAT0.
\end{theorem}

\noindent
We remark that $\D$ could be more general, for example the branch loci
considered by Charney and Davis \cite{CD}.  The current generality is
enough for our applications.

In this section and the next a prime indicates an object in the
branched cover, for example, $\D'$ means $\pi^{-1}(\D)$.  $M$ is
naturally stratified by $\D$, and we write \notation{$M_i$} for the
stratum of codimension~$i$.  This extends our notation $M_0$.  We
write $M_i'$ for $\pi^{-1}(M_i)$.  It is easy to see that each
$M_i'\to M_i$ is a covering map, which we will use implicitly whenever
we lift paths from $M$ to $M'$.

Applying  theorem~\ref{thm-riem-mfld-br-cover-locally-CATx} requires understanding the tangent spaces
$T_{x'}M'$, and the obvious result holds: $T_{x'}M'$ is a branched
cover of $T_xM$.  This shows that the question of its \CAT0-ness is
essentially a problem in piecewise-Euclidean geometry.  This is the
point of the theorem: to reduce a local curvature condition to an
infinitesimal one.

To formulate this precisely, let $x'\in M'$, set $x=\pi(x')$ and
suppose $r>0$ is small enough that the exponential map identifies
$B_r(0)\sset T_xM$ with $B_r(x)$ and $B_r(0)\cap T_x\D$ with
$B_r(x)\cap\D$.  Then the covering map $B_r(x')-\D'\to B_r(x)-\D$
corresponds to a covering space of $B_r(0)-T_x\D$ and therefore to a
covering space of $T_xM-T_x\D$.  We call this last the covering space
of $T_xM-T_x\D$ at $x'$, and its metric completion the branched
covering of $T_xM$ at $x'$. 

\begin{lemma}
\label{lem-tangent-spaces-are-branched-covers-of-tangent-spaces}
$T_{x'}M'$ is the branched covering of $T_xM$ over $T_x\D$ at $x'$. 
\qed
\end{lemma}

\noindent We omit the proof because is an easy application of of the
stratum-wise covering space property.

\begin{proof}[Proof of theorem~\ref{thm-riem-mfld-br-cover-locally-CATx}:]
The ``only if'' assertion is just the fact that a metric space with
curvature bounded above has \CAT0 tangent spaces
\cite[thm.~II.3.19]{bridson-haefliger}.  So we assume all tangent
spaces are \CAT0 and prove $M'$ locally \CATx.  For $x\in M$ define
\notation{$r(x)$} as the supremum of all $r$ such that the following
hold:

(i) The exponential map identifies $B_{3r}(0)\sset T_xM$ with
$B_{3r}(x)\sset M$ and $B_{3r}(0)\cap T_x\D$ with $B_{3r}(x)\cap\D$;

(ii) $B_{3r}(x)$ is convex in $M$;

(iii) if $\x>0$ then $r<\frac{1}{5}\circum X_\x$.

\noindent
We define $D_x$ as $\overline{B_{r(x)}(x)}$, and for $x'\in M'$ lying
over $x$ we define $D'_{x'}$ as $\overline{B_{r(x)}(x')}$.  Obviously
it will suffice to prove the following for all $i$. 

\marginlabel{Claim $\calC_i$}{\it Claim $\calC_i$: for all $x'\in M_i'$, $D'_{x'}$ is \CATx.}  

We
prove this by induction of $i$.  In the base case $i=0$ we are
asserting that $\overline{B_{r(x)}(x')}$ is \CATx\ for any $x'\in
M'-\D$, which holds because it projects isometrically to
$\overline{B_{r(x)}(x)}\sset M$, which is \CATx\ because $M$ has
sectional curvature${}\leq\x$.  See \cite[thm.~II.1A.6]{bridson-haefliger}.

For the inductive step, fix $x'\in M_i'$.  The rest of the proof will
address $D'_{x'}$, so we will write \notation{$D'$} for it and
\notation{$D_j'$} for $D'\cap M_j'$ and similarly for \notation{$D$}
and \notation{$D_j$}.  We will prove $D'$ is \CATx\ by applying
theorem~\ref{thm-CATx-extension} with $X=D'$ and $\D=D_i'$.  To do
this we must verify the assumptions of that theorem.  (If $\x>0$ then
we use theorem~\ref{thm-positive-curvature-case} in place of
theorem~\ref{thm-CATx-extension}.  This is the reason for (iii) above,
which implies $r(x)\leq\frac{1}{5}\circum X_\x$.)

First, $D'$ is complete because it is closed in $M'$.  It is a
geodesic space because $M'$ is
(theorem~\ref{thm-br-cover-is-geodesic-space}, whose proof is
independent of the current theorem) and $D'$ is convex in
$M'$.  To see this convexity, suppose two points $y'$, $z'$ of it are
joined by a geodesic $\gamma'$ in $M'$.  Since $\ell(\gamma')\leq 2r(x)$,
$\gamma'$ lies entirely in $B_{3r}(x')$, so it projects into $B_{3r}(x)$.
If $\pi(\gamma')$ leaves $D$ then we can shorten it by homotoping it (rel
endpoints) along geodesics toward $x$.  Because of the correspondence
between $T_x\D$ and $\D$, this homotopy respects strata, so it lifts
to a homotopy from $\gamma'$ to a shorter path from $y'$ to $z'$.  This is
absurd, so $\pi(\gamma')$ lies in $D$, so $\gamma'$ lies in $D'$, so $D'$ is
convex in $M'$.  A similar argument shows that $D_i'$ is convex in
$D'$.

To prove that condition
\eqref{hyp-radial-geodesic-triangles-CATx} of
theorem~\ref{thm-CATx-extension} holds we follow
\cite[lemmas~I.5.5--6]{CD}.  Suppose $T$
is a geodesic triangle in $D'$ with one vertex in $D_i'$, and call the
opposite edge $E$.  If $E$ lies entirely in one stratum, then we may
subdivide it so that each triangle in the corresponding subdivision of
$T$ projects isometrically into $M$.  Then Alexandrov's lemma shows
that $T$ satisfies \CATx.  Taking limits treats the case in which $E$
lies in one stratum except for its endpoints.  Then another use of
Alexandrov's lemma treats the general case.
Condition~\eqref{hyp-tangent-spaces-unique-local-geodesics} holds 
because all tangent spaces are complete (lemma~\ref{lem-tangent-spaces-are-branched-covers-of-tangent-spaces}) and we are
assuming they are \CAT0.

The real content of the proof is verifying hypothesis
\eqref{hyp-quantified-local-CATx} of
theorem~\ref{thm-CATx-extension}.  For convenience we write
\notation{$|y'|$} for $d(y',D_i')$ when $y'\in D'$, and similarly for
$y\in D$.  We must exhibit $\lambda>0$ such that
$\overline{B_{\lambda|y'|}(y')}$ is \CATx\ for all $y'\in D'-D_i'$.
By our induction hypothesis we know $\overline{B_{r(y)}(y')}$ is
\CATx\ for all $y\in D'-D_i'$, but this is not immediately useful.
The difficulty is that $y'$ can be far away from $D_i'$, yet very
close to a stratum of lower dimension than the one containing $y'$.
Then $r(y)$ is much smaller than $|y|$.  To deal with this situation
we observe that such a $y'$ is very close to a point $z'$ of this
lower-dimensional stratum, around which we will prove by induction on
stratum dimension that there is a fairly large \CATx\ ball.  Since
$y'$ is so close to $z'$, it follows that a smaller but still fairly
large ball around $y'$ is also \CATx.  Our precise statement of this
idea is better stated and proven in $D$ rather than~$D'$:

{\it Claim: Given $j<i$ and $\lambda>0$, there exists $\lambda'>0$
  such that every $y\in D_j$ either has $r(y)\geq\lambda'|y|$ or else
  lies in}
$$
U:=
\bigcup_{z\in Z}B_{\lambda|z|}(z)
\quad\hbox{\it where}\quad
Z=\bigcup_{j<k<i}D_k.
$$ This is easy to prove with $T_xM$ and $T_x\D$ in place of $D$ and
$D\cap\D$, as follows.  The function $y\mapsto r(y)/|y|$ is not
continuous, but its restriction to each stratum is.  Let $K$ be the
unit sphere in $(T_xM_i)^\perp\sset T_xM$, minus its intersection
with $U$.  Let $\lambda'>0$ be a lower bound for the restriction of
$r(y)/|y|$ to the $j$-dimensional stratum in $K$, which exists by
continuity and compactness.  So $r(y)\geq\lambda'|y|$ holds for all
$y\in K$.  It follows that $r(y)\geq\lambda'|y|$ for all $y$ in the
$j$-dimensional stratum of $T_x M-U$, because $r(y)/|y|$ is invariant
under translation in the $T_xM_i$ direction and scaling in the
$(T_xM_i)^\perp$ direction.  ($U$ is also invariant under these
transformations.)  This proves the claim with $T_xM$ and $T_x\D$ in
place of $D$ and $D\cap\D$.  The actual claim follows because the
exponential map $\overline{B_{r(x)}(0)}\to D_x$ is bilipschitz.

Now we are ready to prove that hypothesis
\eqref{hyp-quantified-local-CATx} of
theorem~\ref{thm-CATx-extension} holds.  We use another induction,
to be proven by descending induction on $j=i-1,\dots,0$.  Since $j<i$
we may assume claim $\calC_j$ is known.  Claim $\calD_0$ is exactly
the hypothesis \eqref{hyp-quantified-local-CATx} we want to
verify.

{\it Claim $\calD_j$ $(j=i-1,\dots,0){:}$ there exists $\lambda_j>0$
  such that for all $y'\in D_j'\cup\dots\cup D_{i-1}'$,
  $\overline{B_{\lambda_j|y'|}(y')}$ is \CATx.}

The base case is the largest $j<i$ for which $D_j\neq\emptyset$.  Then
the claim just proven states that there is a positive lower bound for
$r(y)/|y|$ on $D_j$, which we take for our $\lambda_j$.  By claim
$\calC_j$, $\overline{B_{r(y)}(y')}$ is \CATx, and since
$\lambda_j|y'|\leq r(y)$, we have proven $\calD_j$.

Now for the inductive step; suppose $\calD_{j+1}$ is known.  Observe
that if $y'\in D_j'$ lies within $\frac{1}{3}\lambda_{j+1}|z'|$ of some
$z'\in D_{j+1}'\cup\dots\cup D_{i-1}'$, then
$$
|y'|
<
|z'|+\frac{\lambda_{j+1}}{3}|z'|
\leq
2|z'|.
$$
(Obviously we may take $\lambda_{j+1}\leq3$ without loss.)  Therefore
$$
\overline{B_{\lambda_{j+1}|y'|/3}(y')}
\sset
\overline{B_{2\lambda_{j+1}|z'|/3}(y')}
\sset
\overline{B_{\lambda_{j+1}|z'|}(z')}.
$$
The right side is \CATx\ by the inductive hypothesis $\calD_{j+1}$, so
the left side is also.  Now we apply the claim proven above with
$\lambda=\lambda_{j+1}/3$, obtaining $\lambda'$.  Arguing as in the
base case shows that $\overline{B_{\lambda'|y'|}(y')}$ is \CATx\ for
any $y'\in D_j'$ that \emph{doesn't} lie at
distance${}<\frac{1}{3}\lambda_{j+1}|z'|$ from some $z'\in
D_{j+1}'\cup\dots\cup D_{i-1}'$.  So we have proven $\calD_j$ with
$\lambda_j:=\min\{\lambda_{j+1}/3,\lambda'\}$. 
\end{proof}

\begin{remark}
Theorem~5.3 of \cite{CD} is similar to our
theorem~\ref{thm-riem-mfld-br-cover-locally-CATx}.  However, there
is a difficulty with the proof, which we alluded to in
\cite{Allcock-asphericity} and understand better now.  In their
notation, they first show in lemmas~I.5.5--6 that every $\xtilde$ in
the branched cover has a neighborhood $\Utilde$ in which every
``geodesic hinge'' based at $\xtilde$ ``spreads out''.  This is
equivalent to geodesic triangles in $\Utilde$ with a vertex at $\xtilde$
satisfying \CATx.  Then they consider a geodesic triangle in $\Utilde$, use
the fact that the hinges at each of its corners spread out, and appeal
to the equivalence of the hinge-spreading condition with the
\CATx\ condition.

But while the edges at a given corner do diverge in a neighborhood of
that corner, they might converge while still within $\Utilde$.
Another way to say this is that the neighborhoods of the vertices
admitting good local descriptions can be much smaller than $\Utilde$.
Most of theorem~\ref{thm-riem-mfld-br-cover-locally-CATx}'s proof
amounts to wrestling with this issue.  No argument using only the
spreading of every hinge in a neighborhood of its basepoint can prove
the local \CATx\ property, because such an argument would also prove
that our example~\ref{eg-cusped-cone} is \CAT0.
\end{remark}

\section{Branched Covers of Riemannian Manifolds: Global Properties}
\label{sec-Riemannian-global}

\noindent
We continue to use the notation of the previous section, including
$M_i$ and $M_i'$ for strata in $M$ and $M'$.  After proving that $M'$
is a geodesic space (theorem~\ref{thm-br-cover-is-geodesic-space}) we
develop our approach of studying the universal cover of $M_0$ by
relating it to the universal branched cover.
Lemma~\ref{lem-completion-is-homotopy-equivalence} shows that $M_0'\to
M'$ is a homotopy-equivalence under some conditions on the local
topology of the branching.  The main case of interest is when $M_0'$
is the universal cover, but really all that is required is that $M_0'$
be ``locally universal''.  After that we prove our main result on
branched covers of Riemannian manifolds, that $M_0'$ is contractible
and $M'$ is \CATx\ under reasonable hypotheses
(theorem~\ref{thm-riem-mfld-br-cover-globally-CATx}).

\begin{theorem}
\label{thm-br-cover-is-geodesic-space}
Each component of $M'$ is a geodesic space.
\end{theorem}

\begin{proof}
Let $y',z'\in M'$ and let $\gamma_n'$ be a sequence of paths joining
them in $M'$, lying in $M_0'$ except perhaps for their endpoints, with
$\ell(\gamma_n')\to d(y',z')$.  We may suppose they are parameterized
proportionally to arclength.  Project them to paths $\gamma_n$ in $M$,
use the Arzel\`a-Ascoli theorem to pass to a uniformly convergent
subsequence, and write $\gamma$ for the limit.  If $\gamma$ maps some
interval entirely into one stratum, then it is a local geodesic on
that stratum.  Otherwise, for $\gamma_n$ sufficiently close to
$\gamma$, $\gamma_n'$ could be shortened to a path from $y'$ to $z'$
of length${}<\ell(\gamma)=d(y',z')$, a contradiction.  It follows that
every point of $\gamma$ has a neighborhood in $\gamma$
that meets${}\leq3$ strata.  Therefore the domain of $\gamma$ is covered
by finitely many intervals, on the interior of each of which it is a
local geodesic and lies entirely in one stratum.  Now one can see that
for $\gamma_n$ sufficiently close to $\gamma$, we may homotope
$\gamma_n$ to $\gamma$, rel endpoints, such that the homotopy maps
into $M_0$ except perhaps for the endpoints of the paths and the final
path $\gamma$.  Using the covering map $M_0'\to M_0$ and the metric
completion $M_0'\to M'$, we may lift this to a homotopy rel endpoints
from $\gamma_n'$ to some path $\gamma'$ lying over $\gamma$.  This is
the desired geodesic $\overline{y'z'}$.
\end{proof}

\begin{lemma}
\label{lem-completion-is-homotopy-equivalence}
Suppose that for all $x\in M$, 
\begin{enumerate}
\item
\label{hyp-aspherical-local-topology}
$T_xM-T_x\D$ is aspherical, and
\item
\label{hyp-local-pi1s-injective}
the covering space of $T_xM-T_x\D$ at any preimage of $x$ is a universal covering.
\end{enumerate}
Then the
inclusion $M_0'\to M'$ is a homotopy equivalence.
\end{lemma}

\begin{proof}
This is a more sophisticated version of the argument for
\cite[lemma~3.3]{Allcock-asphericity}.  We assume inductively that the
lemma is known for manifolds of dimension smaller than $n:=\dim M$.
The base case $n=1$ is trivial since $\D$ is empty.  Write $\Sigma_i$
for $M_0\cup\dots\cup M_i$ and similarly for $\Sigma_i'$.  We claim
that each inclusion $\Sigma_{i-1}'\to\Sigma_i'$ is a homotopy
equivalence.  This proves the lemma because of the chain
$$
M'
=
\Sigma_n'
\homotopic
\Sigma_{n-1}'
\homotopic
\cdots
\homotopic
\Sigma_0'
=
M'_0.
$$

To prove the claim we must remove $M_i'$ from $\Sigma_i'$ without
changing homotopy type.  So let $x'\in M_i'$ lie over $x\in M_i$.  We
regard $T_x\D$ as a subset (a union of linear subspaces) of $T_xM$,
and by restricting to directions orthogonal to $T_xM_i$ we obtain the
normal bundle $N_\D M_i\sset N_MM_i$.  

For a continuous function
$r:M_i\to(0,\infty)$ we consider the closed-ball-bundle ``$\log B$''
whose fiber over $x\in M_i$ is the closed ball of radius $r(x)$ in
$N_MM_i$.  Standard Riemannian geometry shows that we may choose $r$
such that the exponential map sends $\log B$ diffeomorphically to its
image $B\sset M$ and identifies $(N_\D M_i)\cap(\log B)$ with $\D\cap
B$.  For $0<\alpha\leq1$ we write $\alpha S$ for the image in $M$ of the
sphere-bundle in $N_MM_i$ whose radius at $x$ is $\alpha r(x)$.  In
particular, $S:=1S$ is the boundary of $B$ in $\Sigma_i$.  

We write $B'$ and $\alpha S'$ for the preimages of $B$ and $\alpha S$ in
$\Sigma_i'$, which have obvious projection maps to $M'_{i}$.  Locally
this realizes $S'$ as a fiber bundle over $M'_i$.  (The fibers over
different components of $M_i'$ may be different.)  We claim that every
fiber is contractible.  To see this, let $S'_{x'}$ be a fiber and
$S_x$ its image in $M$.  By hypothesis
\eqref{hyp-local-pi1s-injective}, $S'_{x'}-\D'$ is a copy of the
universal cover of $S_x-\D$, so it is contractible by hypothesis
\eqref{hyp-aspherical-local-topology}.  Furthermore, both
hypotheses \eqref{hyp-aspherical-local-topology} and
\eqref{hyp-local-pi1s-injective} apply with $M$ and $\D$ replaced
by $S_x$ and $S_x\cap\D$, essentially because they are the
``restrictions'' of these hypotheses on $M$ to $\Sigma_{i-1}$.  By
induction on dimension, $S'_{x'}-\D'\to S'_{x'}$ is a
homotopy-equivalence, so $S'_{x'}$ is also contractible.

A standard paracompactness argument shows that we may contract all the
fibers simultaneously, that is, there is a fiber-preserving homotopy
$H$ of $S'$ which contracts each fiber to a point in that fiber.  By
coning off we regard $H$ also as self-homotopy of $B'$.  From this we
derive a new homotopy $J$ of $B'$.  The points of $\alpha S'$ move in $\alpha
S'$ along some initial segment of their tracks under $H$.  This initial
segment is the whole segment for $\alpha\in[0,1/2]$ and then decreases to
the trivial segment (i.e., the constant homotopy) at $\alpha=1$.  By
extending this to the constant homotopy on $\Sigma_i'-B'$ we regard
$J$ as a homotopy of $\Sigma_i'$.  

Note that $J$ collapses a neighborhood of $M_i'$ to a radial
interval-bundle over $M_i'$.  Retracting along these radial intervals
homotopes $M_i'$ into $\Sigma'_{i-1}$.  It is easy to see that this
gives a homotopy-inverse to the inclusion $\Sigma_{i-1}'\to\Sigma_i'$.
\end{proof}

\begin{theorem}
\label{thm-riem-mfld-br-cover-globally-CATx}
Suppose $M$ is connected, $\x\leq0$ and that for all $x\in M$,
\begin{enumerate}
%
%
\item
\label{hyp-AGAIN-T-x-M-T-x-Delta-is-aspherical}
$T_xM-T_x\D$ is aspherical, and
\item
\label{hyp-CAT0-tangent-spaces-in-branched-cover}
the universal branched cover of $T_xM$ over $T_x\D$
is \CAT0.
\end{enumerate}
Then the universal cover $M_0'$ of $M-\D$ is contractible and its
metric completion $M'$ is \CATx.
\end{theorem}

\begin{proof} Without
loss of generality we may replace $M$ by its universal cover and $\D$
by its preimage.  Then $M\iso T_xM$ by the exponential map and
it is easy to see that $\pi_1(T_xM-T_x\D)\to\pi_1(M-\D)$ is injective.
This verifies hypothesis \eqref{hyp-local-pi1s-injective} of lemma~\ref{lem-completion-is-homotopy-equivalence}, and the
hypothesis \eqref{hyp-aspherical-local-topology} of that theorem is the current hypothesis \eqref{hyp-AGAIN-T-x-M-T-x-Delta-is-aspherical}.
So  $M_0'\to M'$ is a homotopy equivalence.  Since
$M_0'$ is simply connected, so is $M'$.
Theorem~\ref{thm-CATx-extension} also applies, so $M'$ is locally
\CATx.  Then the Cartan-Hadamard theorem shows $M'$ is \CATx.  In
particular, it is contractible, and by the homotopy equivalence the
same is true of $M_0'$.
\end{proof}

\begin{remarks}
The only reason we assume connectedness is so the universal cover is
defined.  Also, this proof shows that the ``locally a universal
cover'' hypothesis of
lemma~\ref{lem-completion-is-homotopy-equivalence} is automatic if
$M$ is nonpositively curved.  It is also automatic if $M=S^n$ and $\D$
is a union of great $S^{n-2}$'s; one uses the same argument, together
with the fact that if a point of $S^n$ lies in $\D$ then so does its
antipode.
\end{remarks}

We close this section with a result that simplifies the use of
theorem~\ref{thm-riem-mfld-br-cover-globally-CATx}.  It immediately
implies the main result of \cite{Allcock-asphericity}, namely the case
of theorem~\ref{thm-riem-mfld-br-cover-globally-CATx} with $T_x\D$
locally modeled on the coordinate hyperplanes of $\C^n$ with its usual
metric.  We will use it in a less trivial way in the proof of
lemma~\ref{lem-exceptional-case-lemma}.

\begin{lemma}
\label{lem-can-break-arrangement-into-sub-arrangements}
Suppose $M$ is a complex manifold equipped with a Hermitian Riemannian
metric, and $\D_1,\dots,\D_k$ are locally finite arrangements of
totally geodesic complex hypersurfaces.  Suppose also that every
intersection of a component of $\D_i$ with a component of $\D_{j\neq
  i}$ is Hermitian-orthogonal.  Then $\D=\cup_{i=1}^k\D_i$ satisfies
hypotheses \eqref{hyp-AGAIN-T-x-M-T-x-Delta-is-aspherical} and
\eqref{hyp-CAT0-tangent-spaces-in-branched-cover} of
theorem~\ref{thm-riem-mfld-br-cover-globally-CATx} if each $\D_i$
does.
\end{lemma}

\begin{proof}
By induction it suffices to treat the case $k=2$; let $x\in M$.
Because of the orthogonality, we have
$$
(T_xM-T_x\D_1)\times(T_xM-T_x\D_2)\iso(T_xM-T_x\D)\times\C^{\dim M}.
$$
The left side is aspherical by hypothesis so $T_xM-T_x\D$ is too.
Similarly, the product of the universal branched covers of $T_xM$ over
$T_x\D_1$ and $T_x\D_2$ is isometric to the universal branched cover
over $T_xM-T_x\D$, again times a trivial factor $\C^{\dim M}$.
\end{proof}

\section{Coxeter Arrangements}
\label{sec-Coxeter-arrangements}

\noindent
In this section and the next we apply the machinery we have developed.
Although our results are conditional on the following conjecture about
finite Coxeter groups, we feel the unification of the problems we
address and their reduction to the conjecture is progress in itself.
Our main applications are the Arnol$'$d-Pham-Thom conjecture about
$K(\pi,1)$ spaces for Artin groups
(theorem~\ref{thm-Arnold-Pham-Thom-conjecture}), the asphericity
of the moduli spaces of amply lattice-polarized K3 surfaces
(theorem~\ref{thm-lattice-polarized-K3}), and the asphericity of
discriminant complements for the $3$ kinds of unimodal hypersurface
singularities (section~\ref{sec-singularity-theory}).  Our methods
allow a unified attack on all these problems.  At the end of this
section we make a few remarks on Bridgeland stability conditions for
K3s.

\begin{conjecture}
\label{conj-Coxeter-arrangement-implies-CAT0}
Let $W$ be a finite Coxeter group acting isometrically on
$\C^n$ and let $\D$ be the union of the hyperplanes fixed by the
reflections in $W$.  Then the metric completion of the universal cover
of $\C^n-\D$ is \CAT0.
\end{conjecture}

\begin{remarks}
(1) Most of the applications require only the ADE cases.

(2) The case of a finite complex reflection group is also interesting,
  with some applications we omit here.

(3) The case $n=2$ should be accessible using tools developed by
  Charney and Davis \cite[theorem~9.1]{CD}.  They address the slightly
  different problem of finite-sheeted branched covers of $S^3$ over a
  union of $3$ disjoint great circles.  Our situation is simpler than
  theirs in three ways: our circles lie in a single Hopf fibration, we
  use the universal cover rather than a finite-sheeted cover, and in
  most cases we have more than $3$ great circles.  But I have not
  worked out any details.

(4) This conjecture is very close to conjecture~3 of Charney and Davis
  \cite{CD-artin}.  In particular, we will see in the proof of
  theorem~\ref{thm-Arnold-Pham-Thom-conjecture} that ours implies
  theirs.  The converse may also hold, but I
  have not seriously looked at this.
\end{remarks}

\begin{corollary}
\label{cor-assuming-conjecture-locally-Coxeter-implies-cover-contractible}
Assume conjecture~\ref{conj-Coxeter-arrangement-implies-CAT0}.
Suppose $M$ is a complete connected Riemannian manifold of nonpositive sectional
curvature, and $\D$ is a union of totally geodesic submanifolds, whose
tangent space at any point is isomorphic to the mirror arrangement of
some complexified finite Coxeter group.  Then the universal cover of
$M-\D$ is contractible.
\end{corollary}

\begin{proof}
This is an application of
theorem~\ref{thm-riem-mfld-br-cover-globally-CATx}.  Verifying the
asphericity of $T_xM-T_x\D$ for every $x\in\D$ amounts to the
asphericity of $\C^n$ minus the mirrors of a finite Coxeter group
group.  This is a result of Deligne \cite{Deligne}.  
The \CAT0  hypothesis on the universal branched
cover of $T_xM-T_x\D$ is exactly the conjecture.
(For finite
complex reflection groups one could replace Deligne's theorem with one
of Bessis \cite{Bessis}.)
\end{proof}

\begin{theorem}
\label{thm-Arnold-Pham-Thom-conjecture}
Assume conjecture~\ref{conj-Coxeter-arrangement-implies-CAT0}.  Let
$W$ be any Coxeter group, acting on its open Tits cone $C\sset\R^n$,
and let $M$ be its tangent bundle $TC$.  Let $\D$ be the union of the
tangent bundles to the mirrors of the reflections of $W$.  Then $M-\D$
has contractible universal cover.
\end{theorem}

The theorem applies to many cones besides the Tits cone, but to state
the result in its natural generality one must discuss discrete linear
reflection groups \'a la Vinberg \cite{Vinberg}.  We refer to Charney
and Davis \cite{CD-artin} for the more general formulation; we also
assume familiarity with this paper in the following proof.

\begin{proof}
Charney and Davis \cite[p.~601]{CD-artin} show that the theorem
follows from the claim: the Deligne complex of a finite-type Artin
group is \CAT1.  To explain and prove this statement, suppose $W$ is a
finite Coxeter group, acting in the usual way on $\R^n$, and set
$M:=\C^n$ and $\D$ to be the union of the (complex) mirrors of $W$.
We write $M_0$ for $M-\D$, $M_0'$ for its universal cover, and $M'$
for the metric completion of $M_0'$.  The Artin group associated to
$W$ is defined as $\pi_1(M_0/W)$, and is called finite type to reflect
the fact that $|W|<\infty$.

Formally, its Deligne complex is a metrized simplicial complex defined
in terms of inclusions of cosets of Artin subgroups corresponding to
subdiagrams of the Coxeter diagram.  More convenient for our purposes
is the following: it is the preimage of $S^{n-1}\sset\R^n$ in $M'$,
with the induced path-metric.  To see that this is the same complex
one need only check that face stabilizers are the same as in the Deligne
complex (which is obvious since the Artin subgroups correspond to the
strata of the branch
locus).  The metric in both cases makes each simplex into a copy of the
usual fundamental domain for $W$ in $S^{n-1}$.

The Euclidean cone on the Deligne complex is clearly the preimage
$Y\sset M'$ of $\R^n\sset M$.  By
\cite[thm.~II.3.14]{bridson-haefliger}, a space is \CAT1 if and only
if the cone on it is \CAT0, so it suffices to show that $Y$ is \CAT0.

To prove this we define a distance-nonincreasing retraction $M'\to Y$.
Assuming conjecture~\ref{conj-Coxeter-arrangement-implies-CAT0}, $M'$
is \CAT0, and it then follows that $Y$ is convex in $M'$, hence \CAT0,
proving the theorem.  To define the retraction, consider the homotopy
from $\C^n$ to $\R^n$ given by shrinking the imaginary parts of
vectors toward $0$.  This defines an stratum-preserving ``open
homotopy'' from $M$ into itself, i.e., a continuous map $[0,1)\times
  M\to M$.  Using covering spaces, we can lift this to an open
  homotopy $[0,1)\times M'\to M'$ with $\{0\}\times M'\to M'$ the
    identity map.  (Properly speaking, one lifts the homotopy on each
    stratum separately and checks that they fit together to give a
    homotopy of $M'$.  This is easiest to see by thinking of $M_0$ as
    the set of tangent vectors to $\R^n$ that are not tangent to any
    mirror.  Our homotopy shrinks all vectors without moving their
    basepoints.)

Using metric completeness allows us to to extend this to a homotopy
$[0,1]\times M'\to M'$.  The result is a deformation retraction from
$M'$ to $Y$.  It is distance-nonincreasing because the original
homotopy is.
\end{proof}

The next theorem uses the global Torelli theorem for K3 surfaces,
together with our corollary~\ref{cor-assuming-conjecture-locally-Coxeter-implies-cover-contractible}, to show that
various moduli spaces of K3 surfaces have contractible orbifold
universal covers.  We need the following background on
lattice-polarized K3 surfaces.  These were introduced by Nikulin
\cite{Nikulin-lattice-polarized-K3s} to generalize the classical case of K3 surfaces equipped with
a single ample or semi-ample line bundle.  They were developed further
by Dolgachev \cite{Dolgachev-mirror-symmetry-for-lattice-polarized-K3s}, to which we refer the reader for details.  

Sometimes singular surfaces are called K3s if their minimal
resolutions are K3s, but we include smoothness in the definition of a
K3 surface.  If $X$ is one, then the isometry type of the intersection
pairing on $H^2(X;\Z)$ is independent of $X$ and isomorphic to the
``K3 lattice''
$K:=E_8^2\oplus\bigl(\begin{smallmatrix}0&1\\1&0\end{smallmatrix}\bigr)^3$.
  The Picard group $\Pic X$ is the sublattice spanned by (the
  Poincar\'e duals of) algebraic cycles.  Now suppose $M$ is a lattice
  (=integer bilinear form) of signature $(1,t)$ equipped with a choice
  of Weyl chamber for the subgroup of $\Aut M$ generated by
  reflections in the norm $-2$ vectors of $M$.  An $M$-polarization of
  $X$ means a primitive embedding $j:M\to\Pic X$, and $j$ is called
  ample if the $j$-image of this chamber contains an ample class.  The
  $M$-polarized K3s fall into families parameterized by the isometry
  classes of primitive embeddings $M\to K$, so fix one such embedding.

As concrete examples, when $M$ is spanned by a vector of norm $-4$
(resp.\ $-2$), there is only one embedding $M\to K$ up to isometry.
Then the moduli space of amply $M$-polarized K3s is the same as the moduli
space of smooth quartic surfaces in $\cp^3$ (resp.\ that of the double
covers of $\cp^2$ over smooth sextic curves).

\begin{theorem}
\label{thm-lattice-polarized-K3}
Assume conjecture~\ref{conj-Coxeter-arrangement-implies-CAT0}.
Suppose $M$ is an integer quadratic form of signature $(1,t)$ with a
fixed embedding in $K$.  Then the moduli space of amply $M$-polarized
K3 surfaces $(X,j)$, for which the composition $M\to\Pic X\to H^2(X)$ is
isomorphic to $M\to K$, has contractible (orbifold) universal cover.
\end{theorem}

The proof involves a symmetric space that will also play an important
role in the following section.  If $L$ is a lattice of signature
$(2,n)$ then 
\begin{equation}
\label{eq-def-of-P-Omega}
\Omega(L):= \hbox{a component of } \bigl\{x\in
L\tensor\C\bigm|x\cdot x=0\hbox{ and }x\cdot\bar{x}>0\bigr\}.
\end{equation}
$P\Omega(L)$ is the symmetric space for $\O(L\tensor\R)\iso\O(2,n)$.  (One can
check that $P\Omega(L)$ is the same as the set of positive-definite $2$-planes
in $L\tensor\R$.)  As a symmetric space of noncompact type, its
natural Riemannian metric is complete and has nonpositive sectional curvature.

\begin{proof}
The global Torelli theorem for lattice-polarized K3 surfaces
\cite[theorem~3.1]{Dolgachev-mirror-symmetry-for-lattice-polarized-K3s} says that
the moduli space is covered (as an orbifold) by a certain hyperplane
complement in the symmetric space $P\Omega(M^\perp)$  where $M^\perp$
refers to the complement in $K$.  Namely, it is
$P\Omega(M^\perp)-\D$ where $\D$ is the union of the orthogonal
complements of the norm~$-2$ vectors in $M^\bot\sset K$.  Because the
orthogonal complement of a norm $-2$ vector is the mirror for the
reflection in that vector, and this reflection preserves $M^\perp$, we
see that $\D$ is locally modeled on the hyperplane arrangements for
ADE Coxeter groups.  So we can apply
corollary~\ref{cor-assuming-conjecture-locally-Coxeter-implies-cover-contractible}
(which assumes conjecture~\ref{conj-Coxeter-arrangement-implies-CAT0}).
\end{proof}

I am grateful to D.~Huybrechts for pointing me toward a similar
situation in Bridgeland's work on stability conditions on K3 surfaces.
Bridgeland defined the notion of a stability condition on a
triangulated category, and described one component $\Stab^\dag(X)$ of
the space of locally finite numerical stability conditions on the
derived category of an algebraic K3 surface $X$
\cite[thm.~1.1]{Bridgeland}.  It turns out to be a covering space of
$\Omega\bigl(\Pic
X\oplus\bigl(\begin{smallmatrix}0&1\\1&0\end{smallmatrix}\bigr)\bigr)-\H$
  where $\H$ is the union of the orthogonal complements of the norm
  $-2$ vectors of
  $\Pic{X}\oplus\bigl(\begin{smallmatrix}0&1\\1&0\end{smallmatrix}\bigr)$.
    This situation is exactly like that of theorem~\ref{thm-lattice-polarized-K3}.  In
    particular, our conjecture~\ref{conj-Coxeter-arrangement-implies-CAT0} implies that $\Stab^\dag(X)$ is
    aspherical.  Furthermore, Bridgeland conjectured that this
    covering is a universal covering \cite[conj.~1.2]{Bridgeland}.  If
    both this and our conjecture~\ref{conj-Coxeter-arrangement-implies-CAT0} hold then $\Stab^\dag(X)$ is
    contractible.

\section{Singularity Theory}
\label{sec-singularity-theory}

\noindent
For us, a singularity means the germ $(X_0,x_0)$ of a complex space at
a singular point, and conceptually its semiuniversal deformation (SUD)
is the space of all its smoothings and partial smoothings.  The
discriminant complement means the set of those that are actually
smooth.  Assuming
conjecture~\ref{conj-Coxeter-arrangement-implies-CAT0}. we will prove
that this is aspherical for all the unimodal hypersurface singularities
in Arnol$'$d's hierarchy of surface singularities.

The precise definitions are as follows.  A deformation of $(X_0,x_0)$
means a flat holomorphic map $F:(\frakX,x_0)\to(S,s_0)$ of germs of
complex spaces, together with an isomorphism of $(X_0,x_0)$ with the
fiber over $s_0$.  $F$ is versal if any deformation
$F':(\frakX',x_0)\to(S',s_0')$ can be pulled back from it, i.e., if
there exist holomorphic maps $\phi:S'\to S$ and
$\Phi:\frakX'\to\frakX$ where $F\circ\Phi=F'\circ\phi$ and $\Phi$
respects the identification of the fibers over $s_0$ and $s_0'$ with
$X_0$.  $F$ is called semi-universal if in these circumstances the
derivative of $\phi$ at $s_0$ is always uniquely determined.  Grauert
\cite{Grauert} proved that an isolated singularity always admits a SUD
and that it is unique up to non-unique isomorphism.  So we fix a SUD
$F:(\frakX,x_0)\to(S,s_0)$ and refer to it as ``the'' SUD.  We will
usually speak of spaces rather than germs.

The discriminant $\D$ means the subspace of $S$ over which the fibers
of $\frakX$ are singular.  Studying the inclusion $\D\to S$ and the
topology of $S-\D$ has been at the forefront of singularity theory for
decades, beginning with Brieskorn's famous result
\cite{Brieskorn-ADE}.  He proved that for the $A_n$, $D_n$ and $E_n$
singularities, $\D\to S$ is the inclusion of the mirrors of the
corresponding Weyl group $W$ into $\C^n$, modulo the action of $W$.
In particular, $S-\D$ has $\C^n-(\hbox{the mirrors of $W$})$ as an
unramified covering space.  This connection to singularity theory
helped move Deligne to prove his theorem \cite{Deligne} on the
asphericity of  hyperplane complements like this: it implies that $S-\D$ is
aspherical.

Our goal is to extend Deligne's result to the next level of complexity
in Arnol$'$d's hierarchy of isolated hypersurface singularities: the
``unimodal'' singularities.  Unfortunately, our results are
conditional on
conjecture~\ref{conj-Coxeter-arrangement-implies-CAT0}.  But we do
succeed in unifying the different problems and reducing them to a
question about finite Coxeter groups.  Our arguments also apply to
some other singularities, but we have restricted to the unimodal
hypersurface case to avoid complicated statements.  There are three
flavors of these singularities, all of which occur already for
surfaces in $\C^3$.  It is standard that we lose nothing by
restricting to this dimension \cite[p.~184]{Arnold}.

First come the ``simply elliptic''
singularities $\Etilde_6$, $\Etilde_7$ and $\Etilde_8$, using language
due to K.~Saito \cite{Saito-simply-elliptic}.  These are
\begin{eqnarray}
\nonumber
\Etilde_6&\qquad y(y-x)(y-\lambda x) + xz^2\\
\nonumber
\Etilde_7&\qquad yx(y-x)(y-\lambda x) + z^2\\
\nonumber
\Etilde_8&\qquad y(y-x^2)(y-\lambda x^2) + z^2\\
\noalign{%
\noindent
which are quasihomogeneous of degrees $3$, $4$ and $6$ with respect to
the weights $(1,1,1)$, $(1,1,2)$ and $(1,2,3)$.  Here $\lambda$ must
be such that the singularity is isolated.  For almost all such values
of $\lambda$ one can change coordinates to obtain the more
memorable forms
}
\nonumber
\Etilde_6&\qquad x^3+y^3+z^3+\lambda'xyz\\
\nonumber
\Etilde_7&\qquad x^4+y^4+z^2+\lambda'xyz\\
\nonumber
\Etilde_8&\qquad x^6+y^3+z^2+\lambda'xyz\\
\noalign{%
\smallskip
\noindent
Then come the ``cusp singularities''
}
\label{eq-cusp-Tpqr}
T_{p,q,r}&\qquad x^p+y^q+z^r+xyz
\end{eqnarray}
for $\frac{1}{p}+\frac{1}{q}+\frac{1}{r}<1$.  Finally there are the
``exceptional'' singularities in table~\ref{tab-exceptional-singularities}.  These represent 28
singularities because the cases $\lambda=0$, $\lambda\neq0$ are
essentially different, while rescaling variables allows one to replace
any $\lambda\neq0$ by $\lambda=1$.  When $\lambda=0$, the
singularities are quasihomogeneous with the listed weights and
degrees.  (Note: ``unimodal'' refers to the existence of 1-parameter
families of isomorphism classes of fibers in the SUD, rather than
the number of moduli occurring in the description of the singularity
itself.)

Now let $f(x,y,z)$ be one of the functions above, and
$(X,0)\sset(\C^3,0)$ the singularity defined by $f=0$.  We refer to
\cite{Kas-Schlessinger} for the following standard model of the SUD.
Let $\calO$ be the ring of germs of convergent power series on $\C^3$
at $0$ and $\calI$ be the ideal generated by $f$ and its first partial
derivatives.  (When $f$ is quasihomogeneous, it lies in the ideal
generated by its partial derivatives, so $\calI$ is the Jacobian
ideal.)  Suppose $p_1,\dots,p_\tau\in\calO$ project to a $\C$-basis
for $\calO/\calI$.  Then we define $(S,s_0)$ as $(\C^\tau,0)$,
$(\frakX,0)$ as the subspace of $(\C^{\tau+3},0)$ defined by
\begin{equation}
\label{eq-versal-deformation}
f(x,y,z)+\sum_{i=1}^\tau t_i\, p_i(x,y,z)=0
\end{equation}
and $F$ as the projection that forgets $x,y,z$.  

\begin{table}
\begin{tabular}{ccccc}
&&&&Dolgachev\\
Type&$f$&weights&degree&numbers\\
\noalign{\smallskip}
\hline
\noalign{\smallskip}
$Q_{10}$
&$x^2z+y^3+z^4+\lambda yz^3$
&$9,8,6$
&$24$
&2,3,9
\\
$Q_{11}$
&$x^2z+y^3+yz^3+\lambda z^5$
&$7,6,4$
&$18$
&2,4,7
\\
$Q_{12}$
&$x^2z+y^3+z^5+\lambda yz^4$
&$6,5,3$
&$15$
&3,3,6
\\
$S_{11}$
&$x^2z+yz^2+y^4+\lambda z^3$
&$5,4,6$
&$16$
&2,5,6
\\
$S_{12}$
&$x^2z+yz^2+xy^3+\lambda y^5$
&$4,3,5$
&$13$
&3,4,5
\\
$U_{12}$
&$x^3+y^3+z^4+\lambda xyz^2$
&$4,4,3$
&$12$
&4,4,4
\\
$Z_{11}$
&$x^3y+y^5+z^2+\lambda xy^4$
&$8,6,15$
&$30$
&2,3,8
\\
$Z_{12}$
&$x^3y+xy^4+z^2+\lambda x^4$
&$6,4,11$
&$22$
&2,4,6
\\
$Z_{13}$
&$x^3y+y^6+z^2+\lambda x^4$
&$5,3,9$
&$18$
&3,3,5
\\
$W_{12}$
&$x^4+y^5+z^2+\lambda x^2y^3$
&$5,4,10$
&$20$
&2,5,5
\\
$W_{13}$
&$x^4+xy^4+z^2+\lambda x^3y^2$
&$4,3,8$
&$16$
&3,4,4
\\
$E_{12}$
&$x^3+y^7+z^2+\lambda xy^5$
&$14,6,21$
&$42$
&2,3,7
\\
$E_{13}$
&$x^3+xy^5+z^2+\lambda y^8$
&$10,4,15$
&$30$
&2,4,5
\\
$E_{14}$
&$x^3+y^8+z^2+\lambda xy^6$
&$8,3,12$
&$24$
&3,3,4
\end{tabular}
\medskip
\caption{Arnol$'$d's 14 exceptional singularities.}
\label{tab-exceptional-singularities}
\end{table}

\begin{theorem} 
\label{thm-ellipic-and-exceptional-aspherical}
Assuming conjecture~\ref{conj-Coxeter-arrangement-implies-CAT0}, the discriminant complement of a simply
elliptic or exceptional hypersurface singularity is aspherical.
\end{theorem}

Since the discriminant complement is the complement of one germ inside
another, we clarify: we are asserting that $s_0\in S$ has a basis of
neighborhoods, such that the discriminant complement in each is
aspherical, and any inclusion of one of these discriminant complements
into another is a homotopy equivalence.

\begin{proof}
We postpone the case of non-quasihomogeneous exceptional singularities
to the end of the proof.  Weights and degrees refer to the weighting
of variables given above.  The central object of the proof is the
restriction of $F$ to a certain hypersurface $T$ in $S$.  In every
case one can choose $p_1,\dots,p_\tau$ to be quasihomogeneous, with
all but one of them having degree less than that of $f$.  We take the
exceptional one to be $p_\tau:=$ the part of $f$ involving $\lambda$.
We define $T\sset S$ by $t_\tau=0$.  It turns out that $p_\tau$ is the
Hessian of $f$, up to quotienting by the Jacobian ideal and scaling.
Then a theorem of Wirthm\"uller \cite[Satz~3.6]{Wirthmuller} says that $F$ is
topologically trivial in the $\tau$'th direction.  Precisely: there
are self-homeomorphisms $\Phi$ of $(\C^{3+\tau},0)$ and $\phi$ of
$(S,s_0)$ such that $\phi\circ F=F\circ\Phi$ and
$\Phi(\frakX)=\frakX|_T\times(\C,0)$.  From this we can deduce that
$\phi(\D)=(\D\cap T)\times(\C,0)$.  To see this, note that $\D$ is the
closure of the locus in $S$ over which the fibers of $\frakX$ are not
topological manifolds. (This formulation circumvents any worries about
some fibers being singular as complex hypersurfaces but nonsingular
topologically.)  This shows that $\D$ and $\phi(\D)$ are determined by
the topology of the fibers of $F$ and $F\circ\Phi$, so
$\phi(\D)=(\D\cap T)\times(\C,0)$.  What remains is to prove $T-\D$
aspherical.

Looijenga analyzed $T-\D$ in the simply-elliptic case in
\cite{Looijenga-simply-elliptic-I}.  Take $L$ to be the $E_6$, $E_7$
or $E_8$ root lattice, according to the type of the singularity, and
$W$ the corresponding (finite) Weyl group.  Let $\Lambda\sset\C$ be a
lattice for which $\C/\Lambda$ is isomorphic to the exceptional
divisor in the minimal resolution of $X$ (which is an elliptic curve,
so $\Lambda$ exists).  Both $L\tensor\Lambda\iso\Z^{12,\,14,\,{\rm
    or}\,16}$ and $W$ act on $H:=L\tensor\C$, the former by
translations and the latter via its action on $L$.  Looijenga defines
the ``double affine Weyl group'' $M:=(L\tensor\Lambda)\semidirect W$,
and $\D_H$ as the union of the mirrors of the complex reflections of
$M$.  These are just the $L\tensor\Lambda$-translates of $W$'s
mirrors.  In \cite[Rk.~7.10]{Looijenga-simply-elliptic-II}, he shows
that $T-\D$ has an unramified covering space biholomorphic to
$(H-\D_H)\times\C^*$.  

That is, $(T,s_0)$ has a representative with this property.  By
Looijenga's description, there is a basis of neighborhoods of $s_0$ in
$T$ whose preimages in $H\times\C^*$ have the form
$H\times{}$(exterior of a disk centered at $0$).  So the asphericity
of $(T,s_0)-(\D,s_0)$ reduces to that of $H-\D_H$.  We equip $H$ with
the standard Euclidean metric.  Essentially by definition, $\D_H$ is
locally modeled on finite Coxeter arrangements, so we can apply
corollary~\ref{cor-assuming-conjecture-locally-Coxeter-implies-cover-contractible},
which we recall assumes
conjecture~\ref{conj-Coxeter-arrangement-implies-CAT0}.  This finishes
the proof in the simply-elliptic case.

In the (still quasihomogeneous) exceptional case, there is again an
unramified cover of $T-\D$ which is a $\C^*$-bundle over a hyperplane
arrangement complement.  The details are much more complex, and we
need to present certain of them in order to describe the arrangements
well enough to apply
theorem~\ref{thm-riem-mfld-br-cover-globally-CATx}.  We follow
Looijenga \cite{Looijenga-ICM} for an overview of the construction,
which depends essentially on a method of Pinkham \cite{Pinkham}.  The
whole theory is laid out in much more detail and generality in
\cite{Looijenga-triangle-singularities-II}, and we will use its
description of the hyperplane arrangement rather than the one in
\cite{Looijenga-ICM}.  See also Brieskorn's lovely paper
\cite{Brieskorn-Leopoldina} for a tour of the ideas.  In the notation
of \cite{Looijenga-triangle-singularities-II}, the (quasihomogeneous)
exceptional singularities are $D_{p,q,r}$ triangle singularities,
where $p,q,r$ are the Dolgachev numbers given in
table~\ref{tab-exceptional-singularities}.  Note that
$\frac{1}{p}+\frac{1}{q}+\frac{1}{r}<1$ in every case.

Here is as much detail as we will need.  Pinkham found a
representative for $\frakX|_T$ which is an algebraic family over
$\C^{\tau-1}$.  This is constructed from $\frakX|_T$ by a
quasihomogeneous scaling process, so the discriminant complement in
$(T,s_0)$ is homotopy-equivalent to that in the algebraic family.
So henceforth $T$ will refer to the base of this algebraic family and
$\frakX|_T$ to the total space.  Pinkham also showed that this family
of algebraic surfaces may be simultaneously compactified by adjoining
a suitable divisor.

That is, there is a proper flat family over
$T$ and a divisor therein whose complement is $\frakX|_T$.  This
divisor has $p+q+r-2$ components, each of which meets every fiber in a
smooth rational curve of self-intersection $-2$.  The rational curves
in each fiber meet each other transversally with incidence graph the
$Y_{p,q,r}$ diagram
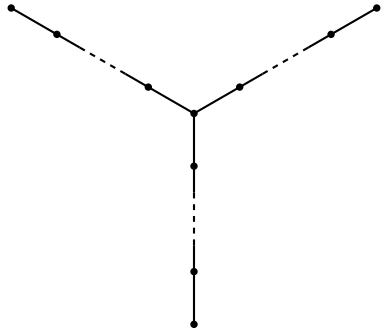
\begin{equation}
\label{eq-Y-p-q-r-graph}
\psset{unit=.0007cm}
\def\noderadius{70}
\begin{pspicture}(-3464,-4000)(3464,2000)
\def\Ax{866}
\def\Ay{500}
\def\Bx{1299}
\def\By{750}
\def\Cx{2165}
\def\Cy{1250}
\def\Dx{2598}
\def\Dy{1500}
\def\Ex{3464}
\def\Ey{2000}
\psset{dash=2pt 2pt}
\qdisk(0,0){\noderadius}
\qline(0,0)(\Bx,\By)
\qline(\Cx,\Cy)(\Ex,\Ey)
\psline[linestyle=dashed](\Bx,\By)(\Cx,\Cy)
\qdisk(\Ax,\Ay){\noderadius}
\qdisk(\Dx,\Dy){\noderadius}
\qdisk(\Ex,\Ey){\noderadius}
\qline(0,0)(-\Bx,\By)
\qline(-\Cx,\Cy)(-\Ex,\Ey)
\psline[linestyle=dashed](-\Bx,\By)(-\Cx,\Cy)
\qdisk(-\Ax,\Ay){\noderadius}
\qdisk(-\Dx,\Dy){\noderadius}
\qdisk(-\Ex,\Ey){\noderadius}
\qline(0,0)(0,-1500)
\qline(0,-2500)(0,-4000)
\psline[linestyle=dashed](0,-1500)(0,-2500)
\qdisk(0,-1000){\noderadius}
\qdisk(0,-3000){\noderadius}
\qdisk(0,-4000){\noderadius}
\end{pspicture}
\end{equation}
where the arms have $p$, $q$ and $r$ vertices including the central
vertex.  Every smooth fiber is a K3 surface.  We write $L$ for
$H^2(\hbox{generic fiber};\discretionary{}{}{}\Z)\iso
E_8^2\oplus\bigl(\begin{smallmatrix}0&1\\1&0\end{smallmatrix}\bigr)^3$,
  $Q$ for the $\Z$-span of these $(-2)$-curves, and $V$ for
  $Q^\perp\sset L$, which turns out to have signature $(2,22-p-q-r)$.
  The fundamental group of $T-\D$ acts on $L$ via its monodromy
  representation, fixing each of the $(-2)$-curves, hence acting on
  $V$.  Write $M$ for its image.  The unramified cover $T'_0$ of
  $T_0:=T-\D$ associated to the kernel of the monodromy representation
  turns out to be $\Omega:=\Omega(V)$, minus the subset lying over a
  hyperplane arrangement in the symmetric space $P\Omega$. ($\Omega(V)$
  was defined in the previous section.)

To quote Looijenga's theorem we must do a little more preparation.
Let $T_f$ be the set of $t\in T$ which admit a neighborhood $U\sset T$
such that the monodromy of $\pi_1(U-\D)$ is a finite group.  The
unramified covering $T'_0\to T_0$ extends naturally to a branched
cover $T'_f\to T_f$.  (This is the normalization of $T'_0$ over
$T_f$, not the current paper's notion of branched cover, although in
the end they are the same.)  Looijenga shows
(\cite[\S5]{Looijenga-ICM} or
\cite[III(6.4)]{Looijenga-triangle-singularities-II}) that $T_f$ is
$M$-equivariantly diffeomorphic to $\Omega\sset V\tensor\C$, minus the
part of $\Omega$ that lies over a certain hyperplane arrangement
$\H_\infty\sset P\Omega$, described in detail in the proof of
lemma~\ref{lem-exceptional-case-lemma} below.  Identifying $T'_f$
with its image in $\Omega$, it turns out that one of its points lies
outside $T'_0$ just if it is orthogonal to a norm $-2$ vector of $V$.
We write $\H_r$ for the corresponding hyperplane arrangement in
$P\Omega$; the $r$ subscript is for ``reflection'', since $M$ contains
the reflections in these vectors.  We summarize the development so
far: $T_0$ has an unramified cover $T'_0$ that is a $\C^*$-bundle
over $P\Omega-\H$, where $\H:=\H_\infty\cup\H_r$.  So the asphericity
of $T_0$ is reduced to that of $P\Omega-\H$, which we establish in
lemma~\ref{lem-exceptional-case-lemma} below (which uses
conjecture~\ref{conj-Coxeter-arrangement-implies-CAT0}).  This finishes the proof in the quasihomogeneous
case.

In the non-quasihomogeneous case, say with $\lambda=1$, the versal
deformation is got from \eqref{eq-versal-deformation} by restricting to the $t_\tau=1$
subspace of the $S$ we used for the corresponding quasihomogeneous
case, call it $\Sigma$.  It turns out that the inclusion
$\D\cap\Sigma\sset\Sigma$ is topologically equivalent to the
inclusion $\D\cap T\sset T$ in the corresponding quasihomogeneous
case, so the asphericity of the discriminant complement reduces to what
we have already proved.  The topological equivalence follows from a
slight strengthening of the topological triviality we used above: the
homeomorphism $\phi:S\to S$ may be taken to preserve the family of
hypersurfaces $t_\tau={\rm constant}$.  
For this we refer to Wirthm\"uller \cite{Wirthmuller}.
\end{proof}

\begin{remark}
It takes some work to extract the required topological triviality
results from the literature.  In the simply-elliptic case Looijenga
\cite{Looijenga-simply-elliptic-I} explicitly states everything we
need.  Wirthm\"uller's Satz~3.6 doesn't quite state what we need,
merely that the SUD is a topological product of the Hessian direction
with some analytic hypersurface $T$.  However, his proof
\cite[p.~63]{Wirthmuller} shows that our $T$ will serve in this role.
Wirthm\"uller's thesis remains unpublished, though a statement appears
in \cite{Damon-topological-triviality-in-versal-unfoldings}.  His work
was later absorbed into a large machine of Damon.  To extract the
results we need from Damon's work, apply
\cite[6.7(ii)]{Damon-finite-determinacy-and-topological-triviality-II}
to $f$ (there called $f_0$) to deduce that $F|_T$ (there called $f$)
is ``finitely $\mathcal{A}$-determined'', and then apply
\cite[Cor.~3]{Damon-finite-determinacy-and-topological-triviality-I}
to deduce that $F$ is a topologically trivial unfolding of $F|_T$.
\end{remark}

\begin{lemma}
\label{lem-exceptional-case-lemma}
In the notation of the previous proof, and assuming
conjecture~\ref{conj-Coxeter-arrangement-implies-CAT0}, $P\Omega-\H$
is aspherical.
\end{lemma}

\begin{proof}
We begin by describing $\H_\infty\sset P\Omega$ where
$\Omega=\Omega(V)$ was defined in section~\ref{sec-Coxeter-arrangements}.  We use the
description from \cite[II\S6]{Looijenga-triangle-singularities-II}.
Recall that $L$ contains the classes of curves intersecting in the
$Y_{p,q,r}$ pattern, $Q$ is their span, and $V=Q^\perp\sset L$ has
signature $(2,22-p-q-r)$.  We attach subscripts $Q$ and $V$ to vectors
to indicate their projections to (the rational spans of) these
lattices.  We define the {\it core} of $Y_{p,q,r}$ to be its
$\Etilde_6$ diagram if it contains one, otherwise its $\Etilde_7$
diagram if it contains one, and otherwise its $\Etilde_8$ diagram.
Let $\E$ be the set of $y\in L$ such that $y^2=-2$, $y\cdot e=1$ for
one end $e$ of $Y_{p,q,r}$ that is not in the core, $y\cdot e'=-1$ for
another end $e'$ also not in the core, and $y$ is orthogonal to all
other roots of $Y_{p,q,r}$.  We will see that $y_V$ has negative norm,
so $y_V^\perp$ defines a hyperplane in $P\Omega$.  Then
$\H_\infty=\cup_{y\in\E}\,y_V^\perp$.  We need to study how these
hyperplanes can meet.  The method is simple: the hyperplanes of
$y,y'\in\E$ meet just if the inner product matrix of $y_V,y'_V$ is
negative-definite.  We can understand this matrix in terms of $y\cdot
y'$ and the inner product matrix of $y_Q,y_Q'$.

The following model of $Q$ is convenient for calculations.  We use
$1+p+q+r$ variables, separated into blocks of sizes $1$, $p$, $q$ and
$r$.  The inner product matrix is
$$
\diag[1;-1,\dots,-1;-1,\dots,-1;-1,\dots,-1]
$$
and the central root has components 
$$
(1;1,0,\dots,0;1,0,\dots,0;1,0,\dots,0).
$$
The roots along the $p$-arm have components
$(-1,1,0,\dots,0),\dots,(0,\discretionary{}{}{}\dots,0,-1,1)$ in the $p$-block and all other
coordinates zero, and similarly for  the $q$- and $r$-arms.  We write $e_p$,
$e_q$ and $e_r$ for the roots corresponding to the end nodes.  
If $t\in\E$ has nonzero inner products with $e_p$ and $e_q$ then we
say it has type $\{e_p,e_q\}$.
All vectors in $Q$ satisfy the $3$ linear relations that
the sum of the coordinates in a block is independent of the block.  If
$y\in\E$ has  $y\cdot e_p=1$ and $y\cdot e_q=-1$, then
one can compute
\begin{equation}
\label{eq-formula-for-y-projected-to-Q}
y_Q=(a;b,\dots,b,b-1;c,\dots,c,c+1;d,\dots,d)
\end{equation}
where
$$
a=\frac{\textstyle\frac{1}{p}-\frac{1}{q}}{\textstyle1-\frac{1}{p}-\frac{1}{q}-\frac{1}{r}}
\qquad
b=\frac{a+1}{p}
\qquad
c=\frac{a-1}{q}
\quad
\hbox{and}
\quad
d=\frac{a}{r}.
$$
We write $N$ for its norm.  Then $y_V^2=-2-N$, which is negative
because calculation reveals
\begin{equation}
\label{eq-formula-for-N+2}
2+N=
\frac{\textstyle\bigl(\frac{1}{p}-\frac{1}{q}\bigr)^2}{\textstyle1-\frac{1}{p}-\frac{1}{q}-\frac{1}{r}}
+\frac{1}{p}+\frac{1}{q}>0.
\end{equation}

Our first claim is that if $y,y'\in\E$ have the same type then their
hyperplanes in $P\Omega$ are disjoint.  For suppose to the contrary
and that both have type $\{e_p,e_q\}$, so that $y_Q$ and $y_Q'$ both
equal \eqref{eq-formula-for-y-projected-to-Q}, after negating $y,y'$ if needed.  Setting $\alpha=y\cdot
y'$, the inner product matrix of $y_V,y_V'$ is
$$
\begin{pmatrix}
-2-N & \alpha-N\\ \alpha-N & -2-N
\end{pmatrix}.
$$
Because the hyperplanes meet, this matrix is negative-definite.  Since
its diagonal entries are negative, this is the same as the determinant
being positive, which boils down to $|\alpha-N|<2+N$.  That is,
$$
-2<\alpha<2+2N.
$$ This is a contradiction because $\alpha\in\Z$ and the right side
is${}\leq-1$.  This inequality is the same as $2+N\leq\frac{1}{2}$,
which is easy to verify in the cases $p=q$ and $(p,q,r)=(4,5,2)$,
where \eqref{eq-formula-for-N+2} reduces to $2/p$ and $1/2$.  (Note:
$p,q\geq4$ because neither $e_p$ nor $e_q$ is in the core.)  The general
case follows from this, $\partial N/\partial q\leq0$ and symmetry in
$p$ and $q$.  We have proven our disjointness claim.

It follows that if a component of $\H_\infty$ and a component of
$\H_r$ meet, then they meet orthogonally.  Otherwise, the reflection
across the latter component would carry the former to another
component of $\H_\infty$ of same type, meeting it nontrivially.  We
use this
orthogonality to break our problem into two simpler problems.
Recall that $P\Omega$ is complete and nonpositively curved.  So to
finish the proof it suffices to verify that $\H$ satisfies hypotheses
\eqref{hyp-AGAIN-T-x-M-T-x-Delta-is-aspherical} and
\eqref{hyp-CAT0-tangent-spaces-in-branched-cover} of
theorem~\ref{thm-riem-mfld-br-cover-globally-CATx}.  By
lemma~\ref{lem-can-break-arrangement-into-sub-arrangements}
it suffices to do this for $\H_r$ and $\H_\infty$ separately.  The
case of $\H_r$ is
corollary~\ref{cor-assuming-conjecture-locally-Coxeter-implies-cover-contractible}
(which uses
conjecture~\ref{conj-Coxeter-arrangement-implies-CAT0}).  So it
suffices to show that for every $x\in P\Omega$,
$T_xP\Omega-T_x\H_\infty$ is aspherical and that the universal
branched cover of $T_xP\Omega$ over $T_x\H_\infty$ is \CAT0.

We have seen that two components of $\H_\infty$ can
meet only if they have different types.  So there is nothing to prove
unless two types are present, which requires that none of the three end
nodes of $Y_{p,q,r}$ lies in the core.  Inspecting the list of
Dolgachev numbers, we see that $p=q=r=4$ is the only case remaining.

So suppose $y\in\E$ projects to \eqref{eq-formula-for-y-projected-to-Q}, and $y'$ similarly with
$p,q,r$ cyclically permuted.  The formulas simplify dramatically since
$p=q=r$, and one finds that $y_Q$ and
$y_Q'$ have norm $-3/2$ and inner product $3/4$.
It follows that $y_V$ and $y_V'$ have inner product matrix
$$
\begin{pmatrix}
-1/2 & \alpha-3/4\\
\alpha-3/4 & -1/2
\end{pmatrix},
$$ and from negative-definiteness that $\alpha=1$.  Therefore
$y''=-y-y'$ lies in $\E$ and has the third type.  We conclude that
where two components of $\H_\infty$ meet, a third does also.  And this
intersection meets no other components of $\H_\infty$ because it
already lies in one of each type.  So $\H_\infty$ is locally modeled
on $y_V^\perp\cup y_V'^\perp\cup y_V''^\perp$ where $y_V$, $y_V'$ and
$y_V''$ are three norm $-1/2$ vectors with sum~$0$.  This is the
hyperplane arrangement for the Coxeter group $W(A_2)$, so we can
complete the proof by appealing to
corollary~\ref{cor-assuming-conjecture-locally-Coxeter-implies-cover-contractible}.
(This part of the proof doesn't really need
conjecture~\ref{conj-Coxeter-arrangement-implies-CAT0}; see the next
remark.)
\end{proof}

\begin{remark}[Triangle singularities]
The appearance of the $A_2$ arrangement is a coincidence arising from
$p=q=r$, but the argument works more generally.  Looijenga
\cite{Looijenga-triangle-singularities-II} treats $D_{p,q,r}$ for
general $p,q,r$, not just the ones in
table~\ref{tab-exceptional-singularities}.  All our arguments go
through, the only differences being that one addresses one smoothing
component of the singularity at a time, and the calculations on the
intersection of components of $\H_\infty$ are messier.  There are only
finitely many cases, because $D_{p,q,r}$ is non-smoothable for
$p+q+r>22$.  A short computer calculation shows $\alpha=1$ in all
cases.  It follows that $\H_\infty$ is locally of the form
$y_V^\perp\cup {y_V'}^\perp\cup{y_V''}^\perp$ where
$y_V+y_V'+y_V''=0$, i.e., three hyperplanes whose intersection has
codimension~$1$ in each of them.  Local asphericity is easy, and the
question of \CAT0-ness of tangent spaces boils down to a study of the
universal cover of $\C^2$ branched over three lines whose
corresponding points in $\cp^1\iso S^2$ lie on a great circle but in
no open hemisphere.  The methods of Charney and Davis
\cite[theorem~9.1]{CD}, adapted to allow infinite branching, apply
here.
\end{remark}

\begin{remark}[The $N_{16}$ singularity]
An $N_{16}$ singularity is the 2-fold cover of $\C^2$ branched over a
union of 5 lines through the origin, and is a trimodal singularity.
Laza has shown \cite[Thm.~5.6]{Laza} that after discarding a
1-dimensional topologically trivial factor, its discriminant
complement has the same form as a triangle singularity.  Namely, it is
a $\C^*$-bundle over $P\Omega(T)-(\H_r\cup\H_\infty)$ for a suitable
lattice $T$ of signature $(2,14)$.  Here $\H_r$ is as above and
$\H_\infty$ is similar to the above.  His $\H_\infty$ is more
complicated than for a triangle singularity and I have not studied it
in any detail.  He has obtained similar but unpublished results for
the $O_{16}$ singularity (the cone on a cubic surface).
\end{remark}

\begin{theorem}
\label{thm-cusp-discriminant-complement-aspherical}
Assuming conjecture~\ref{conj-Coxeter-arrangement-implies-CAT0}, the cusp
singularities \eqref{eq-cusp-Tpqr} have aspherical discriminant complements.
\end{theorem}

\begin{table}
\def\clap#1{\hbox to 0pt{\hss#1\hss}}
\begin{tabular}{ccccc}
       &                  &$\Zykel$                   &$\Zykel^*$\\
$p,q,r$&$c_0,\dots,c_{m-1}$&\clap{$c_0',\dots,c_{m-1}'$}&$d_0',\dots,d_{s-1}'$&\clap{$d_0,\dots,d_{s-1}$}\\
\noalign{\smallskip}
\hline
\noalign{\smallskip}
$2,3,7$&
$1$&
$3$&
$3$&
$1$
\\
$2,4,5$&
$2$&
$4$&
2,3&
$\leftarrow$
\\
$3,3,4$&
$3$&
$5$&
$2,2,3$&
$\leftarrow$
\\
\noalign{\smallskip}
\hline
\noalign{\smallskip}
$2,3,r$&
$3,2^{r-7}$&
$\leftarrow$&
$r-4$&
$r-6$
\\
$2,4,r$&
$4,2^{r-5}$&
$\leftarrow$&
$2,r-2$&
$\leftarrow$
\\
$3,3,r$&
$5,2^{r-4}$&
$\leftarrow$&
$2,2,r-1$&
$\leftarrow$
\\
\noalign{\smallskip}
\hline
\noalign{\smallskip}
$2,q,r$&
$3,2^{q-5},3,2^{r-5}$&
$\leftarrow$&
$q-2,r-2$&
$\leftarrow$
\\
$3,q,r$&
$3,2^{q-4},4,2^{r-4}$&
$\leftarrow$&
$q-1,2,r-1$&
$\leftarrow$
\\
$p,q,r$&
$3,2^{p-4},3,2^{q-4},3,2^{r-4}$&
$\leftarrow$&
$p-1,q-1,r-1$&
$\leftarrow$
\end{tabular}
\medskip
\caption{Data concerning the cusp singularities needed in the proof of
  theorem~\ref{thm-cusp-discriminant-complement-aspherical}.  The
  notation $2^n$ means a string $2,\dots,2$ of $n$ many $2$'s, and
  ``$\leftarrow$'' means ``the same as in the column to the left''.
  We assume $p\leq q\leq r$, and earlier lines take precedence, for
  example the last line applies only when $4\leq p\leq q\leq r$.}
\label{tab-cusp-singularities}
\end{table}

\begin{proof}
Looijenga \cite{Looijenga-rational-surfaces} found a very beautiful
description of $(S,\D)$.  The brief version is that $S-\D$ is
$(\Omega_d-\H)/\Wtilde$, where $\Omega_d$ is the complexified open
Tits cone of the Weyl group $W$ with diagram $Y_{p,q,r}$, an action of
$\Wtilde\iso\Z^{p+q+r-2}\semidirect W$ on $\Omega_d$ is given, and
$\H$ is the locus of points with nontrivial $\Wtilde$-stabilizer.

Suppose $\frac{1}{p}+\frac{1}{q}+\frac{1}{r}<1$ and consider the
singularity \eqref{eq-cusp-Tpqr}.  One can work out a minimal
resolution, and the exceptional divisor turns out to be a union of
rational curves.  These are smooth and meet each other transversely to
form a cycle, except when there is only one component, when it meets
itself transversely at one point.  The negated self-intersection numbers
$c_0,\dots,c_{m-1}$ of these curves are given in
table~\ref{tab-cusp-singularities}
(cf. \cite[(1.3)]{Nakamura-Inoue-surfaces}).  The notation $2^n$ means
a string $2,\dots,2$ of $n$ many $2$'s.  From these one constructs
what Nakamura \cite{Nakamura-Inoue-surfaces} calls the first cycle
numbers $\Zykel(T_{p,q,r})$, for which we write
$(c_0',\dots,c_{m-1}')$.  They are the same as $(c_0,\dots,c_m)$
except when there is only one component, in which case $c_0'=c_0+2$.
Then one works out the ``dual cycle numbers'' $\Zykel^*(T_{p,q,r})$ by
an algorithm due to Hirzebruch and Zagier
\cite[p.~311]{Looijenga-rational-surfaces} and Nakamura
\cite{Nakamura-Inoue-surfaces}, for which we write
$d_0',\dots,d_{s-1}'$.  Then we define $d:=(d_0,\dots,d_{s-1})$ to be
the same sequence, except when there is only one term, when we set
$d_0=d_0'-2$.  (\cite{Looijenga-rational-surfaces} omits mention
of the special treatment of the one-component case.)

One can construct a remarkable surface from these data, a (singular)
hyperbolic Inoue surface
\cite{Inoue}\cite[p.~307]{Looijenga-rational-surfaces}.  It is a
normal but non-algebraic complex surface with two singularities.  One
is a $T_{p,q,r}$ singularity and the other is its ``dual cusp'', the
exceptional divisor $D$ of whose minimal resolution is a cycle of $s$
rational curves with negated self-intersection numbers
$d=(d_0,\dots,d_{s-1})$.  We let $Y$ be the surface obtained by
minimally resolving this dual cusp.  Looijenga explains
\cite[III(2.7)]{Looijenga-rational-surfaces} that a universal
deformation $\Y\to S$ of $Y$ preserving $D$, restricted to the
$T_{p,q,r}$ singularity, gives a semiuniversal deformation of it.
Write $Y_t$ for the fiber over $t\in S$ and $\D\sset S$ for the
discriminant locus $\{t\in S\mid\hbox{$Y_t$ is singular}\}$.  So our
goal is to show the asphericity of $S-\D$.  We will also need to fuss
over the difference between $S-\D$ and $(S,s_0)-(\D,s_0)$.

Looijenga works with a larger space $S_f$, the set of $t\in S$ for
which $Y_t$ has only ADE singularities.  He proves
\cite[III(2.8iii)]{Looijenga-rational-surfaces} that for every $t\in
S_f$, $Y_t$ is a rational surface.  Each such $Y_t$ comes equipped
with a copy $D_t$ of $D$ in it.  In the last paragraph of
\cite[II(3.6)]{Looijenga-rational-surfaces} he defines a complex
manifold $M_d$ and a subset $D_d$, and in
\cite[II(3.7)]{Looijenga-rational-surfaces} he constructs a family of
pairs (rational surface, divisor in it) over $M_d$.  In
\cite[II(3.10)]{Looijenga-rational-surfaces} he describes the singular
fibers of this family; in particular the discriminant is exactly
$D_d$.  Using \cite[II(3.8 and 3.10)]{Looijenga-rational-surfaces} he
shows there is a unique holomorphic map $\Phi_f:S_f\to M_d$ such that
$\Y|_{S_f}$ is the pullback of the family over $M_d$ (up to a possible
minor alteration which in the end doesn't occur; see
\cite[p.~316]{Looijenga-rational-surfaces}).  With $\Phi_f$ in hand,
he extends it \cite[III(3.5)]{Looijenga-rational-surfaces} to a
holomorphic map $\Phi$ from $S$ to a certain completion $\Mhat_d$ of
$M_d$, and shows that this is an isomorphism
\cite[III(3.6)]{Looijenga-rational-surfaces}.  As a consequence
$\Phi_f$ is an isomorphism, so $S-\D\iso M_d-D_d$.  So we want to
prove that $M_d-D_d$ is aspherical.

Here are the definitions of $M_d$ and $D_d$.  Following
\cite[II(3.3)]{Looijenga-rational-surfaces}, let $Y^0$ be a fixed
rational surface with an anticanonical divisor $D^0$ which is a cycle
of type $d$.  Following \cite[I\S2]{Looijenga-rational-surfaces}, let
$Q$ be the subspace of $H_2(Y^0;\Z)$ orthogonal to every component of
$D$, and $B$ a certain explicit basis of $Q$.  One follows Looijenga's
recipe for $B$ and checks that in every case that it is a root system
whose Dynkin diagram $Y_{p,q,r}$ we saw in \eqref{eq-Y-p-q-r-graph}.  It follows
that $Q$ has signature $(1,p+q+r-3)$.  Following
\cite[I\S3]{Looijenga-rational-surfaces} let $W$ be the Weyl group of
$B$ and $I\sset Q\tensor\R$ its Tits cone.  Following
\cite[p.~1]{Looijenga-root-systems}, define $I^\circ$ as the interior
of $I$ and
$$
\Omega_d:=\bigl\{x+iy\in Q\tensor\C\bigm|y\in I^\circ\bigr\}.
$$ (Note: this coincides with \cite[p.~16]{Looijenga-root-systems} for
$X$ the empty subset of $B$.  Also, Looijenga writes $\Omega_d$ in
\cite{Looijenga-rational-surfaces} and $\Omega$ in
\cite{Looijenga-root-systems}.)  Now, $Q$ acts on $\Omega_d$ by
translations in the real directions and $W$ acts by the complexification of
its action on $Q$.  Then $M_d:=\Omega_d/\Wtilde$ where
$\Wtilde:=Q\semidirect W$, and $D_d$ is defined as the image of the
locus $\H\sset\Omega_d$ of points with nontrivial
$\Wtilde$-stabilizer.  Since $M_d-D_d$ has $\Omega_d-\H$ as an
unramified cover, what remains to prove is the asphericity of
$\Omega_d-\H$.

(Properly speaking, $\Omega_d$ is defined in
\cite[II(3.6)]{Looijenga-rational-surfaces} and some unwinding
of definitions in \cite[II(3.2)]{Looijenga-rational-surfaces}
is required to obtain the description above.  The key
points are that $Q$ is nondegenerate, so that $I$ can be identified
with the dual Tits cone $J$, and that choosing a point of the ``affine
lattice'' $A$ identifies it with $Q$.)

Having described $\Omega_d$ we can now address the difference between
$S$ and its germ $(S,s_0)$.  It turns out that for any $W$-invariant
neighborhood $V\sset I$ of $0$, there is a neighborhood $U_V$ of
$s_0\in S=\Mhat_d$ whose preimage in $\Omega_d$ is
$\Utilde_V=\{x+iy\in\Omega_d\mid y\in V\}$.  Furthermore, such $U_V$
give a basis for the topology of $\Mhat_d$ at $s_0$.  (Refer to
\cite[(2.18)]{Looijenga-root-systems} and use the fact that $\{s_0\}$ is the
stratum of $\Mhat_d$ corresponding to the full $Y_{p,q,r}$ diagram,
regarded as a subdiagram of itself.)  Obviously we may restrict to those $V$
that are starshaped around $0$.  For such $V$,
$\Utilde_V-\H\to\Omega_d-\H$ is a homotopy-equivalence.  Being
$\Wtilde$-equivariant, it descends to a homotopy-equivalence
$U_V-\D\to S-\D$.  Therefore the asphericity of $(S,s_0)-(\D,s_0)$ is
equivalent to that of $S-\D$, which we have already reduced to the
asphericity of $\Omega_d-\H$.

At this point the algebraic geometry vanishes into the background,
because $\Omega_d$ and $\H$ are described in terms of $Y_{p,q,r}$.  By
\cite[(2.17)]{Looijenga-root-systems}, point stabilizers in $\Omega_d$
are finite and generated by reflections of $\Wtilde$, so $\H$
is locally modeled on finite Coxeter arrangements.  An obstruction to
applying theorem~\ref{thm-riem-mfld-br-cover-globally-CATx} is the
absence of a nonpositively curved metric on $\Omega_d$.  We can remedy
this as follows.   $W$ is a hyperbolic reflection
group, and $I^\circ$ contains one of the two cones of positive-norm
vectors in $Q\tensor\R$, say $I'^\circ$.  Now,
$$
\Omega_d':=\bigl\{x+iy\in Q\tensor\C \bigm| y\in I'^\circ\bigr\}
\sset \Omega_d
$$ \emph{does} admit a complete nonpositively curved metric.  In the
notation of section~\ref{sec-Coxeter-arrangements}, it is a guise of the symmetric space
$P\Omega\bigl(Q\oplus\bigl(\begin{smallmatrix}0&1\\1&0\end{smallmatrix}\bigr)\bigr)$ for
$\O(2,p+q+r-2)$.  So
corollary~\ref{cor-assuming-conjecture-locally-Coxeter-implies-cover-contractible}
(which assumes conjecture~\ref{conj-Coxeter-arrangement-implies-CAT0})
says that $\Omega_d'-\H$ is aspherical.  To finish the proof we
observe that $\Omega_d'-\H\to\Omega_d-\H$ is a homotopy-equivalence.
To see this, find a $W$-equivariant deformation retraction of
$I^\circ$ into $I'^\circ$ and apply it to the imaginary parts of
points of $\Omega_d$, leaving their real parts fixed.  This is a
$\Wtilde$-equivariant deformation retraction of $\Omega_d$ into
$\Omega_d'$, so $\Omega_d-\H$ and $\Omega_d'-\H$ are
homotopy-equivalent.
\end{proof}

\begin{remark}
In the context of \cite{Looijenga-rational-surfaces} a ``cusp
singularity'' means one whose minimal resolution has exceptional
divisor a cycle of rational curves.  If the cusp is smoothable and the
dual cusp has${}\leq 5$ components, then Looijenga obtained a similar description
of $S$ and $\D$.  Our retraction-to-$\Omega_d'$ trick always works
because the Picard group of $Y^0$ (hence $Q$) still has hyperbolic signature.  Referring to \cite[p.~307]{Looijenga-rational-surfaces} we
see that conjecture~\ref{conj-Coxeter-arrangement-implies-CAT0} implies the asphericity of the discriminant
complement for any 2-dimensional smoothable cusp singularity whose embedding dimension
is${}\leq5$.  Recently Gross, Hacking and Keel \cite{GHK} have
generalized part of \cite{Looijenga-rational-surfaces}, so our methods
may apply even more generally.
\end{remark}


\begin{thebibliography}{99}

\bibitem{alexander-bishop-CH-thm} Alexander, Stephanie; Bishop,
  Richard, The Hadamard-Cartan theorem in locally convex metric
  spaces, {\it L'Enseignement Math\'ematique} {\bf 36} (1990)
  309--320.

\bibitem{alexander-bishop-curved-curves} Alexander, Stephanie;
  Bishop, Richard, Comparison theorems for curves of bounded geodesic
  curvature in metric spaces of curvature bounded above, {\it
    Diff. Geometry. Appl.} {\bf 6} (1996) 67--86.

\bibitem{Allcock-asphericity} Allcock, Daniel, Asphericity of moduli
  spaces via curvature. {\it J. Differential Geom.} {\bf 55} (2000), no. 3,
  441--451.

\bibitem{ACT} Allcock, Daniel; Carlson, James A.; Toledo, Domingo, The
  complex hyperbolic geometry of the moduli space of cubic
  surfaces. {\it J. Algebraic Geom.} {\bf 11} (2002), no. 4, 659--724.

\bibitem{Arnold} Arnolʹd, V. I.; Guseĭn-Zade, S. M.; Varchenko,
  A. N., {\it Singularities of differentiable maps. Vol. I. The
  classification of critical points, caustics and wave
  fronts.} Monographs in Mathematics, 82. Birkhäuser,
  Boston, MA, 1985.

\bibitem{Bessis} Bessis, David, Finite complex reflection arrangements
  are $K(\pi,1)$, {\it arxiv:\discretionary{}{}{}math/0610777}.

\bibitem{Bridson-thesis}
Bridson, Martin R., Geodesics and curvature in metric simplicial
complexes. in {\it Group theory from a geometrical viewpoint (Trieste,
  1990)}, pp.~373--463, World Sci. Publ., River Edge, NJ,
1991. E.~Ghys and P.\ de la Harpe, eds.

\bibitem{bridson-haefliger} Bridson, Martin; Haefliger, Andr\'e,
  {\it Metric spaces of non-positive curvature}, Springer-Verlag,
  Berlin, 1999.

\bibitem{Bridgeland}
Bridgeland, Tom, Stability conditions on $K3$ surfaces, {\it Duke
  Math. J.} {\bf 141} (2008), no. 2, 241--291. 

\bibitem{Brieskorn-ADE} Brieskorn, E., Die Fundamentalgruppe des
  Raumes der regul\"aren Orbits einer endlichen komplexen
  Spiegelungsgruppe. {\it Invent. Math.} {\bf 12} (1971), 57--61.

\bibitem{Brieskorn-Leopoldina} Brieskorn, Egbert, The unfolding of
  exceptional singularities. {\it Leopoldina Symposium: Singularities
  (Thüringen, 1978). Nova Acta Leopoldina (N.F.)} {\bf 52} (1981), no. 240,
  65--93.

\bibitem{CD} Charney, Ruth; Davis, Michael, Singular metrics of
  nonpositive curvature on branched covers of Riemannian manifolds,
  {\it Am. J. Math.} {\bf 115} (1993) 929--1009.

\bibitem{CD-artin} Charney, Ruth; Davis, Michael W., The $K(\pi,1)$-problem for hyperplane
  complements associated to infinite reflection
  groups. {\it J. Amer. Math. Soc.} {\bf 8} (1995), no. 3, 597--627.

\bibitem{Damon-topological-triviality-in-versal-unfoldings} Damon,
  James, Topological triviality in versal unfoldings. {\it Singularities,
  Part 1 (Arcata, Calif., 1981)}, 255--266, Proc. Sympos. Pure Math.,
  40, Amer. Math. Soc., Providence, R.I., 1983.

\bibitem{Damon-finite-determinacy-and-topological-triviality-I} Damon,
  James, Finite determinacy and topological
  triviality I. {\it Invent. Math.} {\bf 62} (1980/81), no. 2, 299--324.

\bibitem{Damon-finite-determinacy-and-topological-triviality-II}
  Damon, James, Finite determinacy and topological
  triviality II. Sufficient conditions and topological
  stability. {\it Compositio Math.} {\bf 47} (1982), no. 2, 101--132.

\bibitem{Deligne} Deligne, Pierre, Les immeubles des groupes de tresses
  g\'en\'eralis\'es. {\it Invent. Math.} {\bf 17} (1972), 273--302.

\bibitem{Dolgachev-mirror-symmetry-for-lattice-polarized-K3s}
  Dolgachev, I. V., Mirror symmetry for lattice polarized $K3$K3
  surfaces. {\it Algebraic geometry, 4. J. Math. Sci.} {\bf 81} (1996), no. 3,
  2599--2630.

\bibitem{Ghys-de-la-Harpe}
Ghys, \'Etienne and Haefliger, Andr\'e, eds., {\it Sur les groupes
  hyperboliques d'apr\`es Mikhael Gromov (Bern, 1988)}, 215--226, {\it
  Progr. Math.} {\bf 83} Birkhäuser Boston, Boston, MA, 1990,

\bibitem{Grauert} Grauert, Hans, \"Uber die Deformation isolierter
  Singularit\"aten analytischer Mengen. {\it Invent. Math.} {\bf 15}
  (1972), 171--198.

\bibitem{Gromov} Gromov, M., Hyperbolic groups. in {\it Essays in
  group theory}, pp.~75--263, {\it Math. Sci. Res. Inst. Publ.}, 8,
  Springer, New York, 1987. S.~Gersten, ed.

\bibitem{GHK} Gross, Mark, Hacking, Paul, and Keel, Sean, Moduli of
  surfaces with an anti-canonical cycle of rational curves, preprint
  2010.

\bibitem{Inoue} Inoue, M., New surfaces with no meromorphic
  functions II. {\it Complex analysis and algebraic geometry,}
  91--106. Iwanami Shoten, Tokyo, 1977.

\bibitem{Kas-Schlessinger} Kas, Arnold; Schlessinger, Michael, On the
  versal deformation of a complex space with an isolated
  singularity. {\it Math. Ann.} {\bf 196} (1972), 23--29.

\bibitem{Laza} Laza, Radu, Deformations of singularities and variation
  of GIT quotients. {\it Trans. Amer. Math. Soc.} {\bf 361} (2009), no. 4,
  2109--2161.

\bibitem{Looijenga-simply-elliptic-I} Looijenga, Eduard, On the
  semi-universal deformation of a simple-elliptic hypersurface
  singularity. Unimodularity. {\it Topology} {\bf 16} (1977), no. 3, 257--262.

\bibitem{Looijenga-simply-elliptic-II} Looijenga, Eduard, On the
  semi-universal deformation of a simple-elliptic hypersurface
  singularity. II. The discriminant. {\it Topology} {\bf 17} (1978), no. 1, 23--40.

\bibitem{Looijenga-ICM} Looijenga, Eduard, Homogeneous spaces
  associated to certain semi-universal deformations. {\it Proceedings of
  the International Congress of Mathematicians (Helsinki, 1978),}
  529--536, Acad. Sci. Fennica, Helsinki, 1980.

\bibitem{Looijenga-triangle-singularities-II} Looijenga, E., The
  smoothing components of a triangle singularity. II. {\it Math. Ann.}
  {\bf 269}
  (1984), no. 3, 357--387.

\bibitem{Looijenga-rational-surfaces} Looijenga, Eduard, Rational
  surfaces with an anticanonical cycle. {\it Ann. of Math. (2)} {\bf 114} (1981),
  no. 2, 267--322.

\bibitem{Looijenga-root-systems} Looijenga, Eduard, Invariant theory
  for generalized root systems. {\it Invent. Math.} {\bf 61} (1980), no. 1, 1--32.

\bibitem{Nakamura-Inoue-surfaces} Nakamura, Iku, Inoue-Hirzebruch
  surfaces and a duality of hyperbolic unimodular
  singularities. I. {\it Math. Ann.} {\bf 252} (1980), no. 3, 221--235.

\bibitem{Nakamura-asphericity} T. Nakamura, A note of the
  K(π,1)-property of the orbit space of the unitary reflection group
  G(m,l,n), {\it Sci. Papers College of Arts and Sciences,
    Univ. Tokyo} {\bf 33}
  (1983), 1--6.

\bibitem{Nikulin-lattice-polarized-K3s}
Nikulin, V.\ V., Finite automorphism groups of K\"ahlerian surfaces of
type K3, {\it Trans. Moscow Math. Soc.} {\bf 38} (1980), No 2, 71--135. 

\bibitem{Orlik-Solomon} P. Orlik, L. Solomon, Discriminants in the
  invariant theory of reflection groups, {\it Nagoya Math. J.} {\bf 109} (1988),
  23--45.

\bibitem{Pinkham} Pinkham, Henry, Deformations of normal surface
  singularities with $\C^*$ action. {\it Math. Ann.} {\bf 232} (1978), no. 1,
  65--84.

\bibitem{Saito-simply-elliptic}
Saito, Kyoji, Einfach-elliptische Singularitäten, {\it Invent. Math.}
{\bf 23} (1974) 289--<325.

\bibitem{van-der-Lek} van der Lek, Harm, Extended Artin groups, in
  {\it Singularities, Part 2 (Arcata, Calif., 1981)} 117--121,
  Proc. Sympos. Pure Math., 40, Amer. Math. Soc., Providence, RI,
  1983.

\bibitem{Vinberg}
Vinberg , \`E. B., 
Discrete linear groups generated by reflections, {\it Math. USSR
  Izv.} {\bf 5} (1971) 1083--1119.

\bibitem{Wirthmuller} Wirthm\"uller, K., {\it Universell topologische
  triviale Deformationen}, Ph.D. thesis, Regensburg, 1978.

\bibitem{Wolpert}
 Wolpert, Scott A,. Geodesic length functions and the Nielsen problem,
 {\it J. Diff. Geom.} {\bf 25} (1987) 275--296.

\bibitem{Wolpert-survey} Wolpert, Scott A., Geometry of the
  Weil-Petersson completion of Teichmüller space, {\it Surveys in
    differential geometry, Vol. VIII} pp.~357–393, Int. Press,
  Somerville, MA, 2003.

\end{thebibliography}
\end{document}